\documentclass{amsart}
\usepackage{amsmath,amsfonts,amssymb,amsthm}
  \usepackage{paralist}
 \usepackage[colorlinks=true]{hyperref}
\hypersetup{urlcolor=blue, citecolor=red}
\usepackage{hyperref}

\usepackage{cite}

\newtheorem{theorem}{Theorem}[section]

\newtheorem{lemma}[theorem]{Lemma}

\theoremstyle{definition}
\newtheorem{definition}[theorem]{Definition}
\newtheorem{remark}{Remark}

\renewcommand\MR[1]{\relax\ifhmode\unskip\spacefactor3000 \space\fi MR~\MRhref{#1}{#1}}

\numberwithin{equation}{section}

\newcommand{\nR}{\mathbb R}
\newcommand{\nZ}{\mathbb Z}
\newcommand{\nT}{\mathbb T}
\newcommand{\nN}{\mathbb N}
\newcommand{\B}{\mathcal{B}}

\newcommand{\vphi}{\varphi}
\newcommand{\tabM}{\textrm{ \:\:\:\:\: }}
\newcommand{\maps}{\rightarrow}

\newcommand{\sand}{\tabM\text{and}\tabM}

\newcommand{\cnj}[1]{\overline{#1}}
\newcommand{\pd}[2]{\frac{\partial #1}{\partial #2}}

\newcommand{\norm}[1]{\left\Vert#1\right\Vert}
\newcommand{\abs}[1]{\left|#1\right|}
\newcommand{\set}[1]{\left\{#1\right\}}
\newcommand{\ip}[2]{\left<#1,#2\right>}
\newcommand{\pnt}[1]{\left(#1\right)}
\newcommand{\pair}[2]{\left(#1,#2\right)}

\def\arXiv#1{
  {\href{http://arxiv.org/pdf/#1}
   {{\mdseries\ttfamily arXiv:#1}}}}

\title[Inviscid Voigt Hydrodynamic Models]
      {On the Higher-Order Global Regularity of the  Inviscid Voigt-Regularization
of Three-Dimensional Hydrodynamic Models}

\author{Adam Larios}
\address[Adam Larios]{Department of Mathematics\\
                University of California, Irvine\\
		Irvine CA 92697-3875, USA}
\email[Adam Larios]{alarios@math.uci.edu}
\author{Edriss S. Titi}
\address[Edriss S. Titi]{Department of Mathematics, and Department of Mechanical and Aero-space Engineering,\\ University of California, Irvine,\\ Irvine CA 92697-3875, USA.\\  Also The Department of Computer Science and Applied Mathematics,\\ The Weizmann Institute of Science, Rehovot 76100, Israel}
\email[Edriss S. Titi]{etiti@math.uci.edu and edriss.titi@weizmann.ac.il}

\thanks{{\bf To appear in:} Discrete and  Continuous Dynamical Systems.}
\begin{document}

\begin{abstract}
We prove higher-order and a Gevrey class (spatial analytic) regularity of solutions to the Euler-Voigt inviscid $\alpha$-regularization of the three-dimensional Euler equations of ideal incompressible fluids.  Moreover, we establish the convergence of strong solutions of the Euler-Voigt model to the corresponding solution of the three-dimensional Euler equations for inviscid flow on the interval of existence of the latter.  Furthermore, we derive a criterion for finite-time blow-up of the Euler equations based on this inviscid regularization.  The coupling of a magnetic field to the Euler-Voigt model is introduced to form an inviscid regularization of the inviscid irresistive magneto-hydrodynamic (MHD) system.  Global regularity of the regularized MHD system is also established.
\end{abstract}

\maketitle
 \thispagestyle{empty}

\textit{ \begin{center}
This work is dedicated to Professor Peter Kloeden\\on the occasion of his 60th birthday.         \end{center}}

\section{Introduction}\label{sec:Int}

One of the outstanding problems in mathematical physics is to find an accurate, practical description of turbulent flows.  This problem is not only out of reach for current mathematical tools, but direct numerical simulation of detailed turbulent flows has proven to be computationally prohibitive.  This is due to the current inability to resolve the wide range of the underlying spatial and temporal scales, even by using the most powerful state-of-the-art computers.  For the case of an incompressible, viscous, homogeneous fluid in a domain $\Omega\subset\nR^3$, the governing equations are widely thought to be given by the Navier-Stokes equations
\begin{subequations}\label{NSE}
\begin{align}
\label{NSE_mom}
\partial_tv +(v\cdot\nabla)v+\nabla p&= \nu\triangle v + f &\text{ in }\Omega\times[0,T),&
\\\label{NSE_con}
\nabla\cdot v &=  0 &\text{ in }\Omega\times[0,T),&
\\\label{NSE_int}
v(0)&= v^{in} &\text{ in }\Omega, \phantom{\times[0,T),}&
\\\label{NSE_bdy}
\text{periodic, or } v&= 0& \text{ on }\partial\Omega.\phantom{\times[0,T)}&
\end{align}
\end{subequations}
where $v(x,t)=(v_1(x,t),v_2(x,t),v_3(x,t))$ denotes the velocity field of the fluid at the point $x=(x_1,x_2,x_3)\in\Omega$ and the time $t\in[0,T)$, $T>0$; $p(x,t)$ denotes the pressure; $f(x,t)=(f_1(x,t),f_2(x,t),f_3(x,t))$ is the body (external) forcing on the fluid and  $\nu>0$ is the kinematic viscosity.  The Dirichlet boundary conditions $v|_{\partial\Omega}=0$ correspond to the physically relevant no-slip case, while the equations under periodic boundary conditions are mathematically simpler, and preserve many (though not all) of the structures present in the Dirichlet case.  In the inviscid case, which is the focus of this paper, the equations are known as the Euler equations.  These equations are identical to \eqref{NSE}, except that $\nu=0$ and the boundary conditions \eqref{NSE_bdy} either remain periodic, or are replaced by the Neumann no penetration  boundary conditions $v\cdot n=0 $ on $\partial\Omega$, where $n$ is the outward point normal vector $\Omega$.  Despite much work on the three-dimensional Euler equations over the last two and a half centuries, many basic questions remain unanswered.  For recent surveys of the known results about the Euler equations, see, e.g., \cite{Bardos_Titi_2007, Constantin2007, Majda_Bertozzi_2002}.

As a result of the unresolved difficulties mentioned above, researchers have focused on obtaining reliable, computable models for the large-scale behavior, by considering, for example, spatial or temporal averages of the Navier-Stokes equations.  However, because of the nonlinear nature of  \eqref{NSE}, it is not possible to obtain an exact closed analytical system for the large-scale motion.  Thus, one must resort to modeling the large-scale motion of the dynamics of  \eqref{NSE} for turbulent flows without having to compute their full-scale dynamics.  In many systems, this is often achieved by filtering, or taking a localized spatial averaging, of the equations for a quantity $\vphi$ to get equations for some filtered quantity $\cnj{\vphi}$.  However, if we apply this method to the non-linear equations \eqref{NSE} (say, under periodic boundary conditions), we arrive at
\begin{subequations}\label{NSEav}
\begin{align}
\label{NSEav_mom}
\partial_t\cnj{v} +\nabla\cdot(\cnj{v\otimes v})+\nabla \cnj{p}&=\nu\triangle \cnj{v} + \cnj{f},
\\\label{NSEav_inc}
\nabla\cdot \cnj{v} &= 0,
\end{align}
\end{subequations}
where we have used the fact that $(v\cdot\nabla) v=\nabla\cdot(v\otimes v)$, due to the incompressibility.  Thus, since $\cnj{v\otimes v}\neq\cnj{v}\otimes \cnj{v}$ for any useful definition of the filtering $\cnj{v}$, \eqref{NSEav} is a system of equations in both $\cnj{v}$ and $v$ (or, equivalently, for the mean $\cnj{v}$ and the fluctuation $v-\cnj{v}$), rather than just $\cnj{v}$ alone.  The fact that we now have more unknowns than equations is known as the closure problem of turbulence.  By adding and subtracting $\nabla\cdot(\cnj{v}\otimes\cnj{v})$ in \eqref{NSEav_mom}, we realize that we must deal with the Reynolds stress tensor,
\begin{equation*}
 \mathcal{R}(v,v):= \cnj{v\otimes v}-\cnj{v}\otimes \cnj{v}.
\end{equation*}
Much effort by researchers modeling turbulence has gone into finding useful approximations to $\mathcal{R}(v,v)$ in terms of $\cnj{v}$ alone.

The approximation $\mathcal{R}(v,v)\approx \cnj{\cnj{v}\otimes \cnj{v}}-\cnj{\cnj{v}}\otimes \cnj{\cnj{v}}$ was introduced and studied by Bardina in \cite{Bardina1980}.  Later, in \cite{Layton_Lewandowski2006}, a simpler approximation was considered, namely $\mathcal{R}(v,v)\approx \cnj{\cnj{v}\otimes \cnj{v}}-\cnj{v}\otimes \cnj{v}$, where in particular the Helmholtz filtering operator $\cnj{\vphi}:= (I-\alpha^2\triangle)^{-1}\vphi$ was used, where $\alpha>0$ is a length scale that represents the width of the spatial filter. (Note that here, the inverse is taken in the context of mean-zero functions with periodic boundary conditions).  This choice of filtering has proven to be very effective in turbulence modeling (see, e.g., \cite{Cao_Lunasin_Titi2006, FoiasHolmTiti2002, IlyinLunasinTiti2006, AlphaClosure1998, Alpha1999, AlphaChannelsPipes1998, CheskidovHolmOlsonTiti2005}), in particular, it is a key ingredient to successes of the $\alpha$-models of turbulence, which were introduced as a contribution to the efforts of modeling large-scale motion.  These analytical models have shown much promise in recent years, notably in the benchmark cases of turbulent flows in pipes and channels \cite{Cao_Lunasin_Titi2006, FoiasHolmTiti2002, IlyinLunasinTiti2006, AlphaClosure1998, Alpha1999, AlphaChannelsPipes1998, CheskidovHolmOlsonTiti2005}.  Namely, analytical solutions to the filtered steady-state equations of the $\alpha$-models were found to successfully match time-averaged experimental data of turbulent flows for a wide range of large Reynolds numbers (see, e.g., \cite{Cao_Lunasin_Titi2006, AlphaClosure1998, Alpha1999, AlphaChannelsPipes1998, CheskidovHolmOlsonTiti2005}).  Using the Helmholtz filtering and the approximation $\mathcal{R}(v,v)\approx \cnj{\cnj{v}\otimes \cnj{v}}-\cnj{v}\otimes \cnj{v}$ in \eqref{NSEav} results in the following model, known as the simplified Bardina model:
\begin{subequations}\label{simp_Bardina}
   \begin{align}
(I-\alpha^2 \triangle)\partial_t u +(u\cdot\nabla)u+\nabla p&= \nu\triangle (I-\alpha^2 \triangle)u + f,
\\\nabla\cdot u &=  0.
\end{align}
\end{subequations}
where $u:=\cnj{v}=(I-\alpha^2\triangle)^{-1}v$.  Note that in the case where $\alpha=0$, \eqref{simp_Bardina}, as a PDE, coincides formally with equations \eqref{NSE} under periodic boundary conditions, a fact which is true for all other $\alpha$-models (see, e.g., \cite{FoiasHolmTiti2002, IlyinLunasinTiti2006, AlphaClosure1998, Alpha1999, AlphaChannelsPipes1998, CheskidovHolmOlsonTiti2005}).  Model \eqref{simp_Bardina} was investigated in \cite{Layton_Lewandowski2006, Cao_Lunasin_Titi2006}.  In \cite{Cao_Lunasin_Titi2006}, the authors proved the global existence and uniqueness of solutions in both the viscous ($\nu>0$) and inviscid ($\nu=0$) cases.  In particular, the authors prove in \cite{Cao_Lunasin_Titi2006} that for initial data $u^{in}\in H^1(\Omega)$, system \eqref{EV}, below, has a unique solution $u\in C^1((-\infty,\infty),H^1(\Omega))$.  This result is of particular interest in the inviscid case since, as of yet, \eqref{EV} is the only $\alpha$-model for which global regularity in the inviscid case has been proven.  Furthermore, it was noted in \cite{Cao_Lunasin_Titi2006} that formally setting $\nu=0$ in \eqref{simp_Bardina} amounts to simply adding the term $-\alpha^2\partial_t \triangle u$ to the left-hand side of \eqref{NSE_mom} (with $\nu=0$), yielding
\begin{subequations}\label{EV}
   \begin{align}
   \label{EV_mom}
 -\alpha^2\partial_t \triangle u+\partial_tu +(u\cdot\nabla)u+\nabla p&=f,
\\\label{EV_con}
\nabla\cdot u &= 0,
\\\label{EV_int}
u(0)&=u^{in},
\end{align}
\end{subequations}
which we call the Euler-Voigt equations.  The boundary conditions are taken to be periodic, and the domain is the periodic unit torus, $\Omega = \nT^3:=[0,1]^3$.  Furthermore, we impose the mean-zero condition,
\begin{equation}\label{mean_zero}
    \int_{\nT^3} u^{in}\,dx=\int_{\nT^3} f\,dx=0,
\end{equation}
which implies that $\int_{\nT^3} u\,dx = 0$.

Remarkably, if one reintroduces the viscous term $\nu\triangle u$ to the right-hand side of \eqref{EV_mom}, the resulting equations happen to coincide with equations governing certain visco-elastic fluids known as Kelvin-Voigt fluids, which were first introduced and studied by A.P. Oskolkov \cite{Oskolkov_1973, Oskolkov_1982}.  These equations are known as the Navier-Stokes-Voigt equations, which is our reason for calling \eqref{EV} the Euler-Voigt equations.  They were proposed in \cite{Cao_Lunasin_Titi2006} as a regularization for either the Navier-Stokes (for $\nu>0$) or Euler (for $\nu=0$) equations, for small values of the regularization parameter $\alpha$.

In the presence of a physical boundary, and under the assumption of the no-slip boundary conditions $u|_{\partial\Omega}=0$, the Navier-Stokes-Voigt equations  (i.e., system \eqref{EV} with $\nu\triangle u$ added to the right-hand side of \eqref{EV_mom}), as an $\alpha$-model regularization for the three-dimensional Navier-Stokes equations, have a very attractive advantage over other $\alpha$-models and subgrid-scale models in that one does not need to impose any additional artificial (i.e.\ non-physical) boundary conditions to prove the global existence and uniqueness of strong solutions, as it has been pointed out in \cite{Cao_Lunasin_Titi2006}.
It is worth mentioning that the long-term dynamics and estimates for the global attractor of the three-dimensional Navier-Stokes-Voigt model was studied in \cite{Kalantarov_Titi_2009}.  Moreover, it was shown recently in \cite{ramos-titi-2009} that the statistical solutions (i.e., invariant probability measures) of the three-dimensional Navier-Stokes-Voigt equations converge, in a suitable sense, to a corresponding statistical solution (invariant probability measure) of the three-dimensional Navier-Stokes equations.
Furthermore, in the context of numerical computations, the Navier-Stokes-Voigt system appears to have less stiffness than the Navier-Stokes system (see, e.g., \cite{Ebrahimi_Holst_Lunasin_2009, LevantRamosTiti2009}).

In \cite{LevantRamosTiti2009}, the statistical properties of the Navier-Stokes-Voigt model were investigated numerically in the context of the Sabra shell phenomenological model of turbulence and were compared with the corresponding Navier-Stokes shell model.  In particular, it was observed that for values of the regularization parameter $\alpha$ smaller than the Kolmogorov dissipation length scale, the structure functions of both models obey the same power laws in the inertial range.  For values of $\alpha$ larger than this scale, two distinct regions associated with the inertial range of the energy spectrum for the Navier-Stokes-Voigt model were observed in \cite{LevantRamosTiti2009}; namely,  a region of low wave numbers obeying the Kolmogorov $k^{−2/3}$ power law, and a region of higher  wave numbers, where energy condensates.

Due to its simplicity, the Voigt $\alpha$-regularization is also well-suited to being applied to other hydrodynamic models, such as the two-dimensional surface quasi-geostrophic equations, demonstrated in \cite{Khouider_Titi2008}, and the three-dimensional magneto-hydrodynamic (MHD) equations, which we demonstrate in this contribution.  It is also worth mentioning that in the case of the inviscid Burgers equation, $u_t+uu_{x}=0$, this type of regularization leads to $-\alpha^2u_{xxt}+u_t+uu_x=0$, which is the well-known Benjamin-Bona-Mahony equation of water waves \cite{Benjamin_Bona_Mahony_1972}.

In the present paper, we focus on the Euler-Voigt equations subject to periodic boundary conditions.  Our results are organized as follows.  In Section \ref{sec:Pre}, well-known results are stated and standard notation is recalled.   In Section \ref{sec:HReg}, we extend the results of \cite{Cao_Lunasin_Titi2006} to prove higher-order  regularity for \eqref{EV} under periodic boundary conditions.  In Section \ref{sec:Gev}, we prove that the solutions enjoy spatial analyticity for $u^{in}$ analytic.  In Section \ref{sec:Blowup}, we prove that solutions to the Euler-Voigt equations converge, in some sense, to sufficiently regular solutions of the Euler equations, as $\alpha\maps0$, on any closed interval of time where the corresponding solutions of the three-dimensional Euler equations exist.  Moreover, we provide a criterion for blow-up of the Euler equations, based on the Voigt $\alpha$-regularization.  In Section \ref{sec:MHD}, we demonstrate the applicability of the Voigt $\alpha$-regularization to other fluid models by considering a similar $\alpha$-regularization of hydrodynamic models involving magnetism.  In particular, we prove the global existence of solutions for the inviscid irresistive MHD-Voigt regularization.  The theorems given in this paper also hold in $\nR^3$ (see, e.g., \cite{Oliver_Titi2001}), but for simplicity we work in $\nT^3$.  Our methods and results readily apply to the viscous (i.e., Navier-Stokes-Voigt) case as well.  For simplicity, we set the forcing $f$ equal to zero in all models considered, although our results hold with little additional effort, given a suitably smooth forcing term satisfying \eqref{mean_zero}.

\section{Preliminaries}\label{sec:Pre}
In this section, we introduce some preliminary material and notations which are commonly used in the mathematical study of fluids, in particular in the study of the Navier-Stokes equations (NSE).  For a more detailed discussion of these topics, we refer to \cite{Constantin_Foias1988,Temam_Functional1995,Temam_Th_Num2001,Foias_Manley_Rosa_Temam2001}.

Let $\mathcal{F}$ be the set of all vector-valued trigonometric polynomials with periodic domain $\nT^3:=[0,1]^3$.  We define the space of `test' functions to be
\[\mathcal{V}:=\set{\vphi\in\mathcal{F}:\nabla\cdot\vphi=0\text{ and  }\int_{\nT^3}\vphi(x)\,dx=0}.\]
We denote by $L^p$ and $H^m$ the usual Lebesgue and Sobolev spaces over $\nT^3$, and define $H$ and $V$ to be the closures of $\mathcal{V}$ in $L^2$ and $H^1$ respectively.  We define the inner products on $H$ and $V$ respectively by
\[(u,v)=\sum_{i=1}^3\int_{\nT^3} u_iv_i\,dx
\sand
((u,v))=\sum_{i,j=1}^3\int_{\nT^3}\pd{u_i}{x_j}\pd{v_i}{x_j}\,dx,
\]
and the associated norms $|u|=(u,u)^{1/2}$, $\|u\|=((u,u))^{1/2}$.  Note that $((\cdot,\cdot))$ is a norm due to the Poincar\'e inequality, \eqref{poincare}, below.
We denote by $V'$ the dual space of $V$.  The action of $V'$ on $V$ is denoted by $\ip{\cdot}{\cdot}\equiv \ip{\cdot}{\cdot}_{V'}$.  Note that we have the continuous embeddings
\begin{equation}\label{embed}
 V\hookrightarrow H\hookrightarrow V'.
\end{equation}
Moreover, by the Rellich-Kondrachov compactness theorem (see, e.g., \cite{Evans1998,Adams2003}), these embeddings are compact for bounded domains.

Let $X$ be a Banach space.  We denote by $L^p((a,b),X)$ the space of Bochner measurable functions $t\mapsto w(t)$, where $w(t)\in X$ for a.e. $t\in(a,b)$, such that the integral $\int_a^b\|w(t)\|_X^p\,dt$ is finite (see, e.g., \cite{Adams2003}). A similar convention is used for $C^k((a,b),X)$.  Abusing notation slightly, we write $w(\cdot)$ for the map $t\mapsto w(t)$.  In the same vein, we often write the vector-valued function $w(\cdot,t)$ as $w(t)$ when $w$ is a function of $x$ and $t$.

We denote by $P_\sigma:L^2\maps H$ the Leray-Helmholtz projection operator (i.e., the orthogonal projection onto solenoidal, i.e., divergence-free, vector spaces), and define the Stokes operator $A:=-P_\sigma\triangle$ with domain $\mathcal{D}(A):=H^2\cap V$.   In our case of periodic boundary conditions, it is known that $A=-\triangle$  (see, e.g., \cite{Constantin_Foias1988,Temam_Functional1995}).  $A^{-1}:H\maps H$ is a positive-definite, self-adjoint, compact operator, and therefore has
an orthonormal basis of eigenfunctions $\vphi_k$ corresponding to a
non-increasing sequence of eigenvalues (see, e.g.,
\cite{Constantin_Foias1988,Temam_Functional1995}).  We observe that $(I+\alpha^2A)^{-1}$ is a well-defined bounded operator and that $(I+\alpha^2A)$ and $(I+\alpha^2A)^{-1}$ are self-adjoint.  Furthermore, due to the periodic boundary conditions, partial derivatives of any order commute with $(I+\alpha^2A)$ and $(I+\alpha^2A)^{-1}$.  We label the
eigenvalues $\lambda_k$ of $A$ so that
$0<\lambda_1\leq\lambda_2\leq\lambda_3\leq\cdots$.  Notice that in
the case of periodic boundary conditions in the torus $\nT^3$ we
have $\lambda_1=(2\pi)^{-2}$. Furthermore, for all $w\in V$, we have
the Poincar\'e inequality
\begin{equation}\label{poincare}
   \|w\|_{L^2(\nT^3)}\leq\lambda_1^{-1/2} \|\nabla w\|_{L^2(\nT^3)}.
\end{equation}
Due to \eqref{poincare}, for $w\in \mathcal{D}(A)$, we have the norm equivalences
\begin{equation}\label{elliptic_reg}
   |A w|\cong\| w\|_{H^2}
   \sand
   \|\nabla w\|_{L^2(\nT^3)}\cong\|w\|.
\end{equation}

It will be convenient to use the standard notation
\begin{equation}\label{Bdef}
 B(w_1,w_2):=P_\sigma((w_1\cdot\nabla)w_2)
\end{equation}
for $w_1,w_2\in\mathcal{V}$.  We list several important properties of $B$ which can be found for example in \cite{Constantin_Foias1988, Foias_Manley_Rosa_Temam2001, Temam_Functional1995, Temam_Th_Num2001}.
\begin{lemma}\label{B:prop}
The operator $B$ defined in \eqref{Bdef} is a bilinear form which can be extended as a continuous map $B:V\times V\maps V'$.  Furthermore, for $w_1$, $w_2$, $w_3\in V$,
\begin{equation}\label{B:Alt}
 \ip{B(w_1,w_2)}{w_3}_{V'}=-\ip{B(w_1,w_3)}{w_2}_{V'},
\end{equation}
and therefore
\begin{equation}\label{B:zero}
 \ip{B(w_1,w_2)}{w_2}_{V'}=0.
\end{equation}

\end{lemma}
\noindent Here and below, $C, c_i$, etc. denote generic constants which may change from line to line.  $C_\alpha, C_\alpha',C(\cdots)$, etc. denote generic constants which depend only upon the indicated parameters.

Next, we recall Agmon's inequality (see, e.g., \cite{Adams2003, Agmon1965, Constantin_Foias1988}).  For $w\in\mathcal{D}(A)$ we have
\begin{equation}\label{Agmon1/2}
 \|w\|_{L^\infty(\Omega)} \leq C\|w\|^{1/2}|Aw|^{1/2}\:\:.
\end{equation}
Finally, we note a result of deRham \cite{XWang1993, Temam_Th_Num2001}, which states that if $g$ is a locally integrable function (or more generally, a distribution), we have
\begin{equation}\label{deRham}
 g =\nabla p \text{ for some distribution $p$ iff } \ip{g}{v}=0\tabM \forall v\in\mathcal{V}.
\end{equation}
Using this theorem, it can be shown (see, e.g., \cite{Constantin_Foias1988,Temam_Th_Num2001}) that \eqref{NSE} with $\nu=0$, $f=0$, and periodic boundary conditions, is equivalent to the functional equation
\begin{equation}\label{fun_Euler}
 \frac{du}{dt}+B(u,u) =0,
\end{equation}
and that \eqref{EV} is equivalent to
\begin{equation}\label{fun_Euler_Voigt}
 (I+\alpha^2A)\frac{du}{dt}+B(u,u) =0,
\end{equation}
where the last equality is understood to hold in the sense of $V'$.  Finally, we define the notion of a  solution to \eqref{fun_Euler_Voigt} which was given in \cite{Cao_Lunasin_Titi2006}.
\begin{definition}\label{def_sol}
   Let $u^{in}\in V$ and consider a time interval $(T_1,T_2)$ with $T_1\leq0<T_2$.  A function $u\in C^1((T_1,T_2),V)$ is said to be a solution to \eqref{fun_Euler_Voigt} if it satisfies \eqref{fun_Euler_Voigt} in the sense of $V'$ and furthermore $u(0)=u^{in}$.
\end{definition}

\section{$H^m$ Regularity}\label{sec:HReg}
In this section we show that solutions to \eqref{EV}, or equivalently \eqref{fun_Euler_Voigt}, are globally well-posed and have as much smoothness as the initial data.  However, we do not expect the solutions to instantaneously gain additional smoothness, even in the viscous (Navier-Stokes-Voigt) case.  The reason for this is that adding the term $-\alpha^2\partial_t \triangle u$ to the Navier-Stokes equations destroys their parabolic structure (indeed, the Navier-Stokes-Voigt equations are well-posed backwards in time \cite{Cao_Lunasin_Titi2006}, a feature which is not present in any parabolic equation).  Instead, the Navier-Stokes-Voigt system behaves like a damped hyperbolic (pseudo-parabolic) system \cite{Kalantarov_Titi_2009}.  Despite this, the Navier-Stokes-Voigt equations (with analytic forcing) possess a finite dimensional global attractor comprised of analytic functions.  For details on these matters, we refer to \cite{Kalantarov-1986, Kalantarov_Titi_2009, Kalantarov_Levant_Titi_2009}.  With these considerations in mind, we now state the following theorem.

\begin{theorem}\label{thm:regularity}
Let $u^{in}\in H^m(\nT^3)\cap V$, for $m\geq 1$.  Then there exists a unique solution $u$ of \eqref{fun_Euler_Voigt} (with $f=0$, for simplicity) with $u\in C^1((-\infty,\infty),H^m(\nT^3)\cap V)$.  Moreover,
\[
\|u(t)\|_{H^m(\nT^3)}\leq C(\alpha,\|u^{in}\|_{H^m})(1+|t|)^{p(m)}.
\]
for all $t\in(-\infty,\infty)$, where $p(1) = 1$, $p(2) = 2$ and $p(m) = 5\pnt{\frac{3}{2}}^{m-3}-1$ for $m\geq3$.
\end{theorem}
We give two proofs of Theorem \eqref{thm:regularity}.  Proof 1 is based on the contraction method, which demonstrates the ODE nature of \eqref{EV}.  Proof 2 uses the Galerkin approximation procedure, which is essential for proving Theorem \ref{thm:gev} below, and furthermore is of interest from the point of view of numerical analysis (see, e.g., \cite{Constantin_Foias1988, Foias_Manley_Rosa_Temam2001, Temam_Functional1995, Temam_Th_Num2001,Marchioro_Pulvirenti_1994} for detailed discussions of this method).

\begin{proof}[Proof 1]  The existence and uniqueness of solutions in the sense of Definition \ref{def_sol} have been established in \cite{Cao_Lunasin_Titi2006}.   We prove only the higher-order regularity.    Applying $(I+\alpha A)^{-1}$ to \eqref{fun_Euler_Voigt} (the inverse taken with respect to the periodic boundary conditions), we obtain
\begin{align}\label{EV_ODE}
 \frac{du}{dt} = (I+\alpha^2A)^{-1}B(u,u):=N(u).
 \end{align}
We claim that $N\!:H^m\cap V\maps H^m\cap V$ is locally Lipschitz continuous for each $m\geq1$.  The case $m=1$ was proven in \cite{Cao_Lunasin_Titi2006}.  We prove the remaining cases $m\geq2$.  Let $u_1,u_2\in H^m\cap V$ be arbitrary and write $\delta u = u_2-u_1$.  We have
\begin{align}
   &\phantom{{}=}\nonumber
   \|N(u_1)-N(u_2)\|_{H^m}
   \\&=\nonumber
   \sum_{0\leq|\beta|\leq m}\sup_{\|\vphi\|_{H^m}=1}
      \pair{\partial^\beta(I+\alpha^2A)^{-1}(B(u_1,u_1)-B(u_2,u_2))}
         {\partial^\beta\vphi}
   \\&=\nonumber
   \sum_{0\leq|\beta|\leq m}\sup_{\|\vphi\|_{H^m}=1}
      \pair{\partial^\beta(B(\delta u,u_1)+B(u_2,\delta u))}
         {\partial^\beta(I+\alpha^2A)^{-1}\vphi}
  \\&=\nonumber
    \sup_{\|\vphi\|_{H^m}=1}
   \sum_{\substack{0\leq|\beta|\leq m\\0\leq\gamma\leq \beta}}
    \binom{\beta}{\gamma}
      \pair{B(\partial^{\gamma}\delta u,\partial^{\beta-\gamma}u_1)
      +B(\partial^{\gamma}u_2,\partial^{\beta-\gamma}\delta u)}
         {\partial^\beta(I+\alpha^2A)^{-1}\vphi}.
\end{align}
In the case $\gamma=0$, we have for $|\beta|\leq m$,
\begin{align}
&\phantom{{}=}\nonumber
   \pair{B(\delta u,\partial^{\beta}u_1)}
         {\partial^\beta(I+\alpha^2A)^{-1}\vphi}
=
   -\pair{B(\delta u,\partial^\beta(I+\alpha^2A)^{-1}\vphi)}
         {\partial^{\beta}u_1}
\\&\leq\nonumber
   \|\delta u\|_{L^\infty(\nT^3)}
   \|\nabla \partial^\beta(I+\alpha^2A)^{-1} \vphi \|_{L^2(\nT^3)}
   \|\partial^{\beta}u_1\|_{L^2(\nT^3)}
\\&\leq\nonumber
C_\alpha\|\delta u\|_{H^{2}(\nT^3)}\|\vphi\|_{H^{m-1}(\nT^3)}\|u_1\|_{H^{m}(\nT^3)}.
\end{align}
Similarly,
\begin{align}
&\phantom{{}=}\nonumber
   \pair{B(u_2,\partial^{\beta}\delta u)}
         {\partial^\beta(I+\alpha^2A)^{-1}\vphi}
\leq
C_\alpha\|u_2\|_{H^{2}(\nT^3)}\|\vphi\|_{H^{m-1}(\nT^3)}\|\delta u\|_{H^{m}(\nT^3)}.
\end{align}
In the case $|\gamma|>0$, we have  for $|\beta|\leq m$, $\gamma\leq\beta$,
\begin{align}
&\phantom{{}=}\nonumber
   \pair{B(\partial^{\gamma}\delta u,\partial^{\beta-\gamma}u_1)}
         {\partial^\beta(I+\alpha^2A)^{-1}\vphi}
\\&\leq\nonumber
\|\partial^{\gamma}\delta u\|_{L^2(\nT^3)}
\|\nabla\partial^{\beta-\gamma}u_1\|_{L^2(\nT^3)}\|\partial^\beta(I+\alpha^2A)^{-1}\vphi\|_{L^\infty(\nT^3)}
\\&\leq\nonumber
C_\alpha\|\delta u\|_{H^{m}(\nT^3)}\|u_1\|_{H^{m}(\nT^3)}\|\vphi\|_{H^{m}(\nT^3)}.
\end{align}
Similarly,
\begin{align}
&\phantom{{}=}\nonumber
   \pair{B(\partial^{\gamma}u_2,\partial^{\beta-\gamma}\delta u)}
         {\partial^\beta(I+\alpha^2A)^{-1}\vphi}
\leq
C_\alpha\|u_2\|_{H^{2}(\nT^3)}\|\delta u\|_{H^{m}(\nT^3)}\|\vphi\|_{H^{m}(\nT^3)}.
\end{align}
Combining the above estimates, we have for $m\geq 2$,
\begin{align}\label{EV_Lip}
   \|N(u_1)-N(u_2)\|_{H^m}
   \leq C_{\alpha,m}\pnt{\|u_1\|_{H^{m}(\nT^3)}+\|u_2\|_{H^{m}(\nT^3)}}\|\delta u\|_{H^{m}(\nT^3)}.
\end{align}
Choose $R$ such that $\|u_1\|_{H^{m}(\nT^3)},\|u_2\|_{H^{m}(\nT^3)}< R$.  Then \eqref{EV_Lip} shows that $N$ is locally Lipschitz in $H^m(\nT^3)\cap V$, with Lipscitz constant $2RC_{\alpha,m}$ in the ball of radius $R$ centered at the origin.  Let $u^{in}\in H^m(\nT^3)$.  By the Picard-Lindel\"of Theorem, there exists a time $T>0$ such that \eqref{EV_ODE} has a solution in $C^1([-T,T],H^m(\nT^3)\cap V)$.  Let $[0,T^*_m)$ be the maximal positive interval of existence.  (For the sake of clarity, we work only in the positive time case, $t\geq0$; however, the same proof holds for $t\leq0$.)  All of the work below takes place on the interval $[0,T^*_m)$.  If $T^*_m<\infty$, then $\limsup_{t\uparrow T^*_m}\|u\|_{H^m}=\infty$, but as we will show below, $u$ is bounded in $C([0,T^*_m),H^m(\nT^3)\cap V)$, and thus we will have $T^*_m=\infty$.    The remainder of the proof is divided into four steps.  The cases $m=1,2,3$ are treated sequentially.  The general case is then shown inductively.

\smallskip

\noindent\textbf{1.} For the case $m=1$, we take the inner product of \eqref{EV_mom} with $u$, integrate by parts, and employ \eqref{deRham} to arrive at
\begin{equation*}
\frac{1}{2}\frac{d}{dt}\pnt{\alpha^2\|\nabla u\|_{L^2(\nT^3)}^2+\|u\|_{L^2(\nT^3)}^2}=0,
\end{equation*}
Integrating this inequality gives
\begin{equation}\label{H1}
\alpha^2\|\nabla u(t)\|_{L^2(\nT^3)}^2+\|u(t)\|_{L^2(\nT^3)}^2
=
\alpha^2\|\nabla u_0\|_{L^2(\nT^3)}^2+\|u_0\|_{L^2(\nT^3)}^2,
\end{equation}
which implies $\|u(t)\|_{H^1(\nT^3)}\leq C_\alpha=C(\alpha,\|u^{in}\|)$, and thus $T^*_1=\infty$.  Note that \eqref{H1} was essentially obtained in \cite{Cao_Lunasin_Titi2006} to show global existence of solutions to \eqref{EV} (equivalently \eqref{fun_Euler_Voigt}).

\smallskip

\noindent\textbf{2.} For the case $m=2$, we take the inner product with $-\triangle u$ and integrate by parts to obtain
\begin{equation}\label{H2e1}
\frac{1}{2}\frac{d}{dt}\pnt{\alpha^2\|\triangle u\|_{L^2(\nT^3)}^2+\|\nabla u\|_{L^2(\nT^3)}^2}=\pair{(u\cdot\nabla)u}{\triangle u},
\end{equation}
where we have used \eqref{deRham}.  We estimate the right-hand side using the H\"older's inequality and Agmon's inequality.
\begin{align*}
 |\pair{(u\cdot\nabla)u}{\triangle u}|
& \leq
\int_{\nT^3} |u||\nabla u| |\triangle u|\,dx
\\& \leq
\|u\|_{L^\infty(\nT^3)}\|\nabla u\|_{L^2(\nT^3)}\|\triangle u\|_{L^2(\nT^3)}
\\& \leq
C\|u\|_{H^1(\nT^3)}^{1/2}\|u\|_{H^2(\nT^3)}^{1/2}
\|\nabla u\|_{L^2(\nT^3)}\|\triangle u\|_{L^2(\nT^3)}
\\& \leq
C\|u\|_{H^1(\nT^3)}^{3/2}\|u\|_{H^2(\nT^3)}^{3/2}
\\& \leq
C_\alpha\|u\|_{H^2(\nT^3)}^{3/2},
\end{align*}
thanks to \eqref{H1}.  Using this estimate in \eqref{H2e1} and employing \eqref{elliptic_reg} gives
\begin{equation}\label{H2e2}
\frac{d}{dt}\pnt{\alpha^2\|\triangle u\|_{L^2(\nT^3)}^2+\|\nabla u\|_{L^2(\nT^3)}^2}
\leq
C_\alpha\pnt{\alpha^2\|\triangle u\|_{L^2(\nT^3)}^2+\|\nabla u\|_{L^2(\nT^3)}^2}^{3/4}.
\end{equation}
Thus, using Gr\"onwall's inequality and the norm equivalence \eqref{elliptic_reg}, we have the algebraic growth rate in time
\begin{equation}\label{H2growth}
 \|u(t)\|_{H^2(\nT^3)}\leq
C_\alpha\pnt{\|u^{in}\|_{H^2(\nT^3)}^{1/2}+t}^2,\text{ for all }t\in[0,T].
\end{equation}
Thus, $T^*_2=\infty$.
\smallskip

\noindent\textbf{3.} For the case $m=3$, we take the inner product with $\triangle^2 u$, integrate by parts and use \eqref{deRham} to obtain
\begin{equation}\label{H3e1}
\frac{1}{2}\frac{d}{dt}\pnt{\alpha^2\|\nabla\triangle u\|_{L^2(\nT^3)}^2+\|\triangle u\|_{L^2(\nT^3)}^2} =-\pair{(u\cdot\nabla)u}{\triangle^2 u}.\end{equation}
For the right-hand side, we integrate by parts twice to obtain
{\allowdisplaybreaks
\begin{align*}
&\phantom{{}=}\abs{\int_{\nT^3} [(u\cdot\nabla) u]\cdot\triangle^2 u\,dx}
\\&=
\abs{\int_{\nT^3} [\triangle((u\cdot\nabla) u)]\cdot\triangle u\,dx}
\\&=
\abs{\int_{\nT^3} \Big[(\triangle u\cdot\nabla) u+(u\cdot\nabla)\triangle u +2\sum_{i=1}^3\partial_i u\cdot\nabla\partial_i u\Big]\cdot\triangle u\,dx}
\\&=
\abs{\int_{\nT^3} [(\triangle u\cdot\nabla) u]\cdot\triangle u \,dx
+ 2\sum_{i=1}^3\int_{\nT^3} (\partial_i u\cdot\nabla\partial_i u)\cdot\triangle u\,dx}
\\&\leq
\|\triangle u\|_{L^2({\nT^3})}^2\|\nabla u\|_{L^\infty({\nT^3})}
\\&\;\;\;+
2\|\nabla u\|_{L^\infty({\nT^3})}\|\nabla\nabla u\|_{L^2({\nT^3})}\|\triangle u\|_{L^2({\nT^3})}
\\&\leq
C\|u\|_{H^2({\nT^3})}^2\|\nabla u\|_{L^\infty({\nT^3})}
\\&\leq
C\|u\|_{H^2({\nT^3})}^2\|\nabla u\|_{H^1({\nT^3})}^{1/2} \|\nabla u\|_{H^2({\nT^3})}^{1/2}
\\&\leq
C\| u\|_{H^2({\nT^3})}^{5/2} \| u\|_{H^3({\nT^3})}^{1/2}
\end{align*}
}
on $[0,T]$.  For the third equality above, we have employed the fact that \\ $\int_{\nT^3} ((u\cdot\nabla)\triangle u)\cdot \triangle u\,dx=0$, which follows from \eqref{B:zero} since $\nabla\cdot u=0$.  The first inequality is due to H\"older's inequality, and the third is due to Agmon's inequality \eqref{Agmon1/2}.
Using the above estimate and \eqref{H2growth} in \eqref{H3e1} gives
\begin{equation}\label{H3e2}
\frac{d}{dt}\pnt{\alpha^2\|\nabla\triangle u\|_{L^2(\nT^3)}^2+\|\triangle u\|_{L^2(\nT^3)}^2}
\leq
C'_\alpha\pnt{C_\alpha t+\|u^{in}\|_{H^2({\nT^3})}^{1/2}}^5\|u\|_{H^3({\nT^3})}^{1/2}.
\end{equation}
For the last factor on the right-hand side of \eqref{H3e2}, one can use the norm equivalence \eqref{elliptic_reg} to show that
\begin{equation}\label{H3_equiv}
\|u\|_{H^3({\nT^3})}^{1/2}\leq C_\alpha\pnt{\alpha^2\|\nabla\triangle u\|_{L^2(\nT^3)}^2+\|\triangle u\|_{L^2(\nT^3)}^2}^{1/4}.
\end{equation}
Combining \eqref{H3e2} and \eqref{H3_equiv}, we may apply Gr\"onwall's inequality and again the norm equivalence \eqref{elliptic_reg} to find the algebraic growth rate in time
\begin{align}\label{H3growth}
\nonumber \|u\|_{H^3({\nT^3})}
&\leq\Big[
\|u^{in}\|_{H^3({\nT^3})}^{3/2}+C''_\alpha(C_\alpha t +\|u^{in}\|_{H^2({\nT^3})})^6
\Big]^{2/3}
\\&\leq C_\alpha\Big[
\|u^{in}\|_{H^3({\nT^3})}+(t+\|u^{in}\|_{H^2({\nT^3})})^4
\Big].
\end{align}
Thus $T^*_3=\infty$.

\smallskip

\noindent\textbf{4.} For the remaining cases $m\geq 4$, we work by induction on $m$.  That is, we show that if the $H^{m-1}$ norm is bounded, then so is the $H^m$ norm.  We follow closely the techniques used in \cite{Marchioro_Pulvirenti_1994} (see also \cite{Majda_Bertozzi_2002, Temam1976}).  Let $\beta$ be a multi-index such that $|\beta|\leq m-1$.  We apply $\partial^\beta$ to both sides of \eqref{EV_mom}, take the inner product of \eqref{EV_mom} with $\partial^\beta u$ and use \eqref{deRham} to obtain
\begin{align}\label{Hme1}
\nonumber &\phantom{={}}\frac{1}{2}\frac{d}{dt}
\pnt{\alpha^2\|\nabla \partial^\beta u\|_{L^2({\nT^3})}^2+\|\partial^\beta u\|_{L^2({\nT^3})}^2}
\\\nonumber& =
-\pair{\partial^\beta((u\cdot\nabla)u)}{\partial^\beta u}.
\\\nonumber& =
-\sum_{0\leq\gamma\leq\beta}\binom{\beta}{\gamma}\int_{\nT^3}
[(\partial^\gamma u\cdot\nabla)\,\partial^{\beta-\gamma}u]\cdot\partial^\beta u\,dx
\\& =
-\sum_{0<\gamma\leq\beta}\binom{\beta}{\gamma}\int_{\nT^3}
[(\partial^\gamma u\cdot\nabla)\,\partial^{\beta-\gamma}u]\cdot\partial^\beta u\,dx
\end{align}
since in the case $\gamma=0$ we have $\int_{\nT^3}
[(u\cdot\nabla)\,\partial^{\beta}u]\cdot\partial^\beta u\,dx=0$, again by the divergence free condition \eqref{B:zero}.  We estimate the integrals in \eqref{Hme1}, analyzing the cases $|\gamma|=1$, $|\gamma|=2$ and $|\gamma|\geq 3$ separately.

For $|\gamma|=1$, $\gamma\leq\beta$, we have
\begin{align*}
 \int_{\nT^3} [(\partial^\gamma u\cdot\nabla)\,\partial^{\beta-\gamma}u]\cdot\partial^\beta u\,dx
&\leq
\|\partial^\gamma u\|_{L^\infty({\nT^3})}
\|\nabla\partial^{\beta-\gamma} u\|_{L^2({\nT^3})}
\|\partial^\beta u\|_{L^2({\nT^3})}
\\&\leq
C\|u\|_{H^{|\gamma|+2}}\|u\|^2_{H^{|\beta|}}
=
C\|u\|_{H^3}\|u\|^2_{H^{|\beta|}}.
\end{align*}

For $|\gamma|=2$, $\gamma\leq\beta$, we have
\begin{align*}
 \int_{\nT^3} [(\partial^\gamma u\cdot\nabla)\,\partial^{\beta-\gamma}u]\cdot\partial^\beta u\,dx
&\leq
\|\partial^\gamma u\|_{L^6({\nT^3})}
\|\nabla\partial^{\beta-\gamma} u\|_{L^3({\nT^3})}
\|\partial^\beta u\|_{L^2({\nT^3})}
\\&\leq
C\|\partial^\gamma u\|_{H^1({\nT^3})}
\|\nabla\partial^{\beta-\gamma} u\|_{H^1({\nT^3})}
\|\partial^\beta u\|_{L^2({\nT^3})}
\\&\leq
C\|u\|_{H^3({\nT^3})}
\|u\|_{H^{|\beta|-1}({\nT^3})}
\|u\|_{H^{|\beta|}({\nT^3})}
\\&\leq
C\|u\|_{H^3({\nT^3})}
\|u\|_{H^{|\beta|}({\nT^3})}^2.
\end{align*}

For $|\gamma|\geq3$, $\gamma\leq\beta$, we have
\begin{align*}
 \int_{\nT^3} [(\partial^\gamma u\cdot\nabla)\,\partial^{\beta-\gamma}u]\cdot\partial^\beta u\,dx
&\leq
\|\partial^\gamma u\|_{L^2({\nT^3})}
\|\nabla\partial^{\beta-\gamma} u\|_{L^\infty({\nT^3})}
\|\partial^\beta u\|_{L^2({\nT^3})}
\\&\leq
C\| u\|_{H^{|\gamma|}({\nT^3})}
\|\nabla\partial^{\beta-\gamma} u\|_{H^2({\nT^3})}
\|\partial^\beta u\|_{L^2({\nT^3})}
\\&\leq
C\| u\|_{H^{|\gamma|}({\nT^3})}
\|u\|_{H^{|\beta|}({\nT^3})}^2.
\end{align*}
These estimates imply that the right-hand side of \eqref{Hme1} is bounded above by \\ $C(\|u\|_{H^3({\nT^3})} +\|u\|_{H^{|\beta|}({\nT^3})})\|u\|_{H^{|\beta|}({\nT^3})}^2 \leq C\|u\|_{H^{m-1}({\nT^3})}^3$.  Using this with \eqref{Hme1} and summing over all $\beta$ with $|\beta|\leq m-1$ gives
\begin{equation}\label{Hme2}
\frac{d}{dt}\pnt{\alpha^2\|\nabla u\|_{H^{m-1}({\nT^3})}^2+\|u\|_{H^{m-1}({\nT^3})}^2}\leq C\|u\|^3_{H^{m-1}({\nT^3})}.
\end{equation}
Integrating this inequality and using the norm equivalence \eqref{elliptic_reg}, we find the following recursive relationship for $m\geq4$:
\begin{equation}\label{HmgrowthRec}
\|u(t)\|_{H^{m}({\nT^3})}
\leq C_\alpha\pnt{  \|u^{in}\|^2_{H^{m}({\nT^3})}+\int_0^t\|u(s)\|^3_{H^{m-1}({\nT^3})}\,ds
}^{1/2}.
\end{equation}
Thus $T^*_m=\infty$.  Repeatedly iterating estimate \eqref{HmgrowthRec} and eventually using \eqref{H3growth}, we see that $\|u(t)\|_{H^{m}({\nT^3})}$ grows at most algebraically in $t$, that is
\begin{equation}\label{Hmgrowth}
 \|u(t)\|_{H^{m}({\nT^3})}
\leq C(\alpha,\|u^{in}\|_{H^m({\nT^3})})(1+t)^{p(m)},\text{ for all }t\in[0,\infty).
\end{equation}
for some algebraic growth rate $p(m)$.  Using induction and \eqref{HmgrowthRec}, we have   $p(m)=\pnt{3p(m-1)+1}/2$ for $m\geq 4$.  From $\eqref{H3growth}$, we have $p(3)=4$.  Solving this difference equation, we find $p(m) = 5\pnt{3/2}^{m-3}-1$, as claimed.
\end{proof}

   \begin{proof}[Proof 2]
   For each $N\in\nN$, let $P_N$ denote the orthogonal projection in $H$ given by $P_N:H\maps\text{span}\set{\vphi_k}_{|k|\leq N}:=H_N$ ($\vphi_k$ are defined in Section \ref{sec:Pre}).  The Galerkin approximation to \eqref{fun_Euler_Voigt} at level $N$ is given by the following system:
   \begin{subequations}\label{Gal_EV}
   \begin{align}
   \label{Gal_EV_mo}
 (I+\alpha^2 A)\frac{d}{dt}u_N&=-P_NB(u_N,u_N),
\\
\label{Gal_EV_IC}
u_N(0)&=P_N u^{in}.
\end{align}
\end{subequations}
Applying the operator $(I+\alpha^2 A)^{-1}$ to \eqref{Gal_EV_mo}, we see that this is an ODE in the finite-dimensional space $H_N$ with quadratic non-linearity, and therefore has a solution $u_N\in C^1((-T_N,T_N),H_N)$ for some $T_N>0$.  Thus, taking the inner product of \eqref{Gal_EV_mo} with $u_N$ and integrating by parts is justified, and we arrive at
\begin{equation}\label{H1_Galerkin}
\alpha^2\|\nabla u_N(t)\|_{L^2(\nT^3)}^2+\|u_N(t)\|_{L^2(\nT^3)}^2
\leq
\alpha^2\|\nabla u_0\|_{L^2(\nT^3)}^2+\|u_0\|_{L^2(\nT^3)}^2.
\end{equation}
This implies that $T_N=\infty$ for all $N\geq1$.   Let $T>0$ be fixed but arbitrary.  For simplicity, we work on the interval $[0,T]$, but the same proof holds on $[-T,0]$.  Following ideas similar to the case of the 3D Euler equations (see, e.g., \cite{Marchioro_Pulvirenti_1994}), we show that  $\set{u_N}_{N\in\nN}$ is a Cauchy sequence in $C([0,T],V)$.  Let $u^{in}\in V$, and let $u_N$ and $u_M$ be solutions of \eqref{Gal_EV} for $N,M\in\nN$, $N<M$ respectively, and consider $w_M^N:=u_N-u_M\in H_M$.  We will need the following inequalities, which follow easily from Parseval's identity.
\begin{align*}
   \|u_N\|_{H^2(\nT^3)}\leq CN \|u_N\|_{H^1(\nT^3)}
   \quad\text{and}\quad
   \|(I-P_N)u_M\|_{L^2(\nT^3)}\leq CN^{-1}\|u_M\|_{H^1(\nT^3)}.
\end{align*}
Subtracting the equations \eqref{Gal_EV_mo} for $u_N$ and $u_M$, applying $I+\alpha^2 A$ to both sides, taking the inner product with $w_M^N$ and using \eqref{B:Alt} and \eqref{B:zero}, we find
\begin{align*}
  &\phantom{{}=}
   \frac{1}{2}\frac{d}{dt}\pnt{\|w_M^N\|_{L^2(\nT^3)}^2
   +\alpha^2\|\nabla w_M^N\|_{L^2(\nT^3)}^2}
   \\&=
   \pnt{B(u_N,u_N),(I-P_N)u_M}-\pnt{B(w_M^N,u_N),w_M^N}
   \\&\leq
   C\|u_N\|_{H^1(\nT^3)}^{3/2}\|u_N\|_{H^2(\nT^3)}^{1/2}\|(I-P_N)u_M\|_{L^2(\nT^3)}
   +C\|w_M^N\|_{H^1(\nT^3)}^2\|u_N\|_{H^1(\nT^3)}
   \\&\leq
   C\|u_N\|_{H^1(\nT^3)}^{3/2}N^{1/2}\|u_N\|_{H^1(\nT^3)}^{1/2}N^{-1}\|u_M\|_{H^1(\nT^3)}
   +C\|w_M^N\|_{H^1(\nT^3)}^2\|u_N\|_{H^1(\nT^3)}
   \\&\leq
   C_\alpha N^{-1/2}+C_\alpha\pnt{\|w_M^N\|_{L^2(\nT^3)}^2
   +\alpha^2\|\nabla w_M^N\|_{L^2(\nT^3)}^2},
\end{align*}
since $\|u_N\|_{H^1(\nT^3)}\leq C_\alpha$ uniformly in $N$.  Gr\"onwall's inequality yields
\begin{align*}
&\phantom{{}=}
\|w_M^N(t)\|_{L^2(\nT^3)}^2
   +\alpha^2\|\nabla w_M^N(t)\|_{L^2(\nT^3)}^2
   \\&\leq
   \pnt{\|w_M^N(0)\|_{L^2(\nT^3)}^2
   +\alpha^2\|\nabla (w_M^N(0))\|_{L^2(\nT^3)}^2}e^{C_\alpha t}+C_\alpha N^{-1/2}(e^{C_\alpha t}-1).
   \end{align*}
Since $u^{in}\in V$,  it follows that $\|w_M^N(0)\|_{L^2(\nT^3)}^2
   +\alpha^2\|\nabla (w_M^N(0))\|_{L^2(\nT^3)}^2\maps 0$ as $M,N\maps \infty$.   Therefore, the above inequality shows that $\set{u_N}_{N=1}^\infty$ is a Cauchy sequence and hence converges to an element $u\in C([0,T],V)$.  Choose $v\in \mathcal{V}$ arbitrarily.  Since $u_N\in C([0,T],\mathcal{V})$, taking the inner product of \eqref{Gal_EV_mo} with $v$ and integrating in time is justified, so we have
\begin{align}
   &\phantom{{}=}\nonumber
   (u_N(t),v)+\alpha^2((u_N(t),v)) - (u_N(0),v)-\alpha^2((u_N(0),v))
   \\&=\label{Gal_EV_Int}
   -\int_0^t(B(u_N(s),u_N(s)),P_Nv)\,ds.
\end{align}
Since $u_N\maps u$ strongly in $C([0,T],V)$, we have $(u_N(t'),v)\maps (u(t'),v)$, and\\ $((u_N(t'),v))\maps ((u(t'),v))$ for all $t'\in[0,T]$.  Furthermore, we note that
\begin{align*}
   &\phantom{{}=}\int_0^t(B(u_N(s),u_N(s)),P_Nv)\,ds
   -\int_0^t\ip{B(u(s),u(s))}{v}\,ds
   \\&=
   \int_0^t\pair{B(u_N(s),u_N(s))}{P_Nv-v}\,ds
   +\int_0^t\ip{B(u_N(s)-u(s),u_N(s))}{v}\,ds
      \\&\phantom{{}=}
   +\int_0^t\ip{B(u(s),u_N(s)-u(s))}{v}\,ds.
\end{align*}
Using this and the facts that $\|u_N\|_{H^1(\nT^3)}$ is bounded independantly of $N$, that  $u_N\maps u$ strongly in $C([0,T],V)$, and that $\|P_Nv-v\|_{H^1(\nT^3)}\maps0$ as $N\maps\infty$, we find that $\int_0^t(B(u_N(s),u_N(s)),P_Nv)\,ds \maps\int_0^t\ip{B(u(s),u(s)}{v}\,ds$ as $N\maps\infty$.  Passing to the limit in \eqref{Gal_EV_Int}, we have
\begin{align*}
   (u(t),v)+\alpha^2((u(t),v)) - (u^{in},v)-\alpha^2((u^{in},v))
   =
   -\int_0^t\ip{B(u(s),u(s))}{v}\,ds
\end{align*}
 for all $v\in\mathcal{V}$.  A simple argument using the density of $\mathcal{V}$ in $V$ shows that  the equality holds for all $v\in V$.  Since $u\in C([0,T],V)$ we have $Au,\, B(u,u)\in C([0,T],V')$.  Therefore the above equality shows that \eqref{fun_Euler_Voigt} holds in the sense of $V'$ and that $u(0)=u^{in}$.    Thus, we have proven the existence of a solution to \eqref{fun_Euler_Voigt}.   Thanks to \eqref{fun_Euler_Voigt} and the fact that $(I+\alpha^2A)^{-1}B(u,u)\in C([0,T],V)$, we have $\frac{d}{dt} u\in C([0,T],V)$ so that $u\in C^1([0,T],V)$.  To show the uniqueness and continuous dependence on initial data, let $u_1, u_2\in C([0,T],V)$ are two solutions of \eqref{fun_Euler_Voigt} initial data  $u_1^{in},u_2^{in}\in V$ respectively, and we write $\delta u:=u_2-u_1$.  Subtracting, we have
 \begin{align}\label{w_eqn}
    \frac{d}{dt}\delta u = (I+\alpha^2A)^{-1}(-B(u_2,\delta u)-B(\delta u,u_1)).
 \end{align}
 Since $\frac{d}{dt}\delta u\in C([0,T],V)$, we may justifiably take the inner product of \eqref{w_eqn} with $\delta u$ and move the time derivative outside the inner product.  Using \eqref{B:Alt} and the fact that $(I+\alpha^2A)^{-1}$ is self-adjoint, we have
  \begin{align*}
    \frac{1}{2}\frac{d}{dt}\|\delta u\|_{L^2(\nT^3)}^2
    &=
    \pair{B(u_2,(I+\alpha^2A)^{-1}\delta u)}{\delta u}
    -\pair{B(\delta u,u_1)}{(I+\alpha^2A)^{-1}\delta u}
    \\&\leq
     \|u_2\|_{L^6(\nT^3)}\|\nabla(I+\alpha^2A)^{-1}\delta u\|_{L^3(\nT^3)}\|\delta u\|_{L^2(\nT^3)}
     \\&\phantom{{}=}+
     \|\delta u\|_{L^2(\nT^3)}\|\nabla u_1\|_{L^2(\nT^3)}\|(I+\alpha^2A)^{-1}\delta u\|_{L^\infty(\nT^3)}
     \\&\leq
     C_\alpha\pnt{\|u_2\|_{H^1(\nT^3)}+\|u_1\|_{H^1(\nT^3)}}\|\delta u\|_{L^2(\nT^3)}^2
     \leq C_\alpha\|\delta u\|_{L^2(\nT^3)}^2,
 \end{align*}
 since $\|u_1\|_{H^1(\nT^3)},\|u_2\|_{H^1(\nT^3)}\leq C_\alpha$.  By Gr\"onwall's inequality, we see that\\ $\|u_2(t)-u_1(t)\|_{L^2(\nT^3)}^2 \leq \|u_2^{in}-u_1^{in}\|_{L^2(\nT^3)}^2e^{C_\alpha t}$.  In particular, if $u_1^{in}=u_2^{in}$, then $u_2(t)=u_1(t)$ for all $t\in [0,T]$.

Finally, we show the higher-order regularity.  Suppose that $u^{in}\in H^m(\nT^3)\cap V$.  The process leading to \eqref{H2growth}, \eqref{H3growth}, and \eqref{Hmgrowth} is justified for  $u_N$, since $u_N\in C([0,T],\mathcal{V})$.  Thus, $u_N$ is bounded uniformly with respect to $N$ in $C([0,T],H^m(\nT^3)\cap V)$.  By the Banach-Alaoglu theorem, we see that $u_N$ has a subsequence converging in the weak-$*$ topology to an element $v\in L^\infty([0,T],H^m(\nT^3)\cap V)$, and furthermore that the estimates \eqref{H2growth}, \eqref{H3growth}, and \eqref{Hmgrowth} hold for a.e. $t\in[0,T]$.  On the other hand, we already know that $u_N\maps u$ in $C([0,T],V)$.  Therefore $u=v\in L^\infty([0,T],H^m(\nT^3)\cap V)$.  From \eqref{fun_Euler_Voigt}, it is easy to see that $\frac{d}{dt}v\in L^\infty([0,T],H^m(\nT^3)\cap V)$, and thus by the Sobolev embedding theorem, we have $u\in C([0,T],H^m(\nT^3)\cap V)$.  Again by \eqref{fun_Euler_Voigt}, we see that $u\in C^1([0,T],H^m(\nT^3)\cap V)$.
\end{proof}

\section{A Gevrey Class Regularity (Spatial Analyticity)}\label{sec:Gev}
Next, we show that \eqref{EV} (or equivalently \eqref{fun_Euler_Voigt}) has additional regularity, under the assumption of the relevant regularity of  $u^{in}$.  Namely we show that it has a specific type of Gevrey regularity which is analytic in space.  In accordance with the discussion in Section \ref{sec:HReg}, we do not expect solutions for initial data which is not analytic to become analytic instantaneously, unlike in the case of parabolic equations.  (However, see  \cite{Kalantarov_Levant_Titi_2009}, which proves the analytic regularity of the attractor of the Navier-Stokes-Voigt equations, i.e., in the presence of viscosity and an analytic forcing term.  See also \cite{Paicu_Vicol_2009} for similar results concerning second-grade fluids.)  The concept of Gevrey regularity was first used in the context of the Navier-Stokes equations in \cite{Foias_Temam1989} and was expanded to more general non-linear parabolic equations in \cite{Ferrari_Titi1998} (see also \cite{Cao_Rammaha_Titi_1999, Cao_Rammaha_Titi_2000}).  For the three-dimensional Navier-Stokes-Voigt equations, Gevrey regularity for the attractor was shown in \cite{Kalantarov_Levant_Titi_2009}.  Gevrey regularity for the Euler equations in bounded domians  has been studied in \cite{Bardos_Benachour1977} and \cite{Kukavica_Vicol_2009}.

We define the Gevrey class of order $s>0$ for a given $r\geq0$ as the domain of a particular class of operators parameterized by $\tau>0$:
\begin{equation*}
 \mathcal{D}(A^r e^{\tau A^{1/(2s)}}):=\set{u\in H^r(\nT^3)\cap V:\|e^{\tau A^{1/(2s)}}u\|_{H^r(\nT^3)}<\infty}.
\end{equation*}
For our purposes, we will work in the case $s=1$, where $\mathcal{D}(e^{\tau A^{1/2}})$ corresponds to the set of real analytic functions $C^\omega(\nT^3)$.  For more on Gevrey classes, see \cite{Ferrari_Titi1998, Foias_Manley_Rosa_Temam2001, Foias_Temam1989, Levermore_Oliver1997, Rodino-1993, Kukavica_Vicol_2009}.  We will need the following lemma, contained in \cite{Levermore_Oliver1997} (see also \cite{Oliver_Titi2000}).
\begin{lemma}\label{lemma:LivermoreOliver}
 For $\psi,\vphi$ defined on  $\nT^n$ such that   $\psi\in\mathcal{D}(A^{r/2+1/4}e^{\tau A^{1/2}})$ and $\vphi\in\mathcal{D}(A^{r/2+1/2}e^{\tau A^{1/2}})$, where $r>\frac{n+3}{2}$ and $\nabla\cdot \psi=0$, we have
\begin{equation}\label{NonLinEst}
\nonumber\begin{split}
&\phantom{={}}\abs{\pair{A^{r/2}e^{\tau A^{1/2}}B(\psi, \vphi)}{A^{r/2}e^{\tau A^{1/2}} \vphi}}
\\&\leq
c_1\|A^{r/2} \psi\|_{L^2}\|A^{r/2} \vphi\|^2_{L^2}
+
\tau c_2\left(
    \|A^{r/2} e^{\tau A^{1/2}}\psi\|_{L^2}\|A^{r/2+1/4}e^{\tau A^{1/2}} \vphi\|^2_{L^2}\right.
\\&\phantom{={}}+
    \left.\|A^{r/2+1/4}e^{\tau A^{1/2}} \psi\|_{L^2}\|A^{r/2} e^{\tau A^{1/2}} \vphi\|_{L^2}\|A^{r/2+1/4}e^{\tau A^{1/2}} \vphi\|_{L^2}
    \right)
\end{split}
\end{equation}
where $c_1$, $c_2>0$ depend only on $r$.
\end{lemma}

We are now ready to state and prove the main theorem of this section.  The proof given in the present work is established rigorously using the Galerkin method.   To establish our estimates, we follow some of the ideas of  \cite{Levermore_Oliver1997}, which were applied formally to the 2D lake equations to establish formal \textit{a priori} estimates.  However, the results of \cite{Levermore_Oliver1997} can be justified rigorously by using methods similar to those employed in the proof below.

\begin{theorem}\label{thm:gev}
   Let $u^{in}\in \mathcal{D}(A^{(r+1)/2} e^{\sigma A^{1/2}})$ for some $r>3$, $r\in\nN$, and $\sigma\geq0$.  Then there exists a unique solution $u$ of \eqref{fun_Euler_Voigt} (with $f=0$, for simplicity) satisfying
\begin{align}
&\phantom{{}=}\label{Gev_Bnd_Thm}
\alpha^2\|A^{r/2}e^{\tau(t) A^{1/2}}\nabla u(t)\|^2_{L^2({\nT^3})}
+
\|A^{r/2}e^{\tau(t) A^{1/2}} u(t)\|^2_{L^2({\nT^3})}
\\&\leq\nonumber
\alpha^2\|A^{r/2}e^{\sigma A^{1/2}}\nabla u^{in}\|^2_{L^2({\nT^3})}
+
\|A^{r/2}e^{\sigma A^{1/2}} u^{in}\|^2_{L^2({\nT^3})}
+
2c_1\int_0^{|t|}\kappa^3(\xi)\,d\xi.
\end{align}
   for all $t\in(-\infty,\infty)$, where $\tau(t):=\sigma\exp\pnt{-c_2\int_0^{|t|} h(\xi)\,d\xi}$, $h>0$ is defined by \eqref{h_def} below, $c_1, c_2>0$ are given by Lemma \ref{lemma:LivermoreOliver} above, and $\kappa(\xi):=C_\alpha(1+|\xi|)^{p(r)}$, where $p(r) = 5\pnt{\frac{3}{2}}^{r-3}-1$.
\end{theorem}

The above theorem shows that for all $t\in(-\infty,\infty)$, $u(\cdot,t)$ is bounded in the space $\mathcal{D}(A^{(r+1)/2}e^{\tau(t)A^{1/2}})$.  Therefore, $u$ is analytic with respect to the spatial variable.  Furthermore, the radius of analyticity of the solution, which is bounded from below by $\tau(t)$, may be shrinking as time increases, but it never collapses to zero in finite time,  since $\tau(t)>0$ for all $t$.

\begin{proof}  First, we observe that the case $\sigma=0$ is already covered by Theorem \ref{thm:regularity}; therefore, we assume that $\sigma>0$.  Choose $T>0$ arbitrarily.  Here again, we focus on the interval $[0,T]$.  The proof for $[-T,0]$ is the same.  Clearly $\mathcal{D}(A^{(r+1)/2} e^{\sigma A^{1/2}})\subset H^{r+1}(\nT^3)$, so the conclusions of Theorem \ref{thm:regularity} hold.  It remains to prove \eqref{Gev_Bnd_Thm}, which we establish using  the Galerkin approximation method.  Consider the system \eqref{Gal_EV}.  As discussed in Proof 2 of Theorem \ref{thm:regularity}, this system has a solution $u_N$ in $C^1((-\infty,\infty),H_N)$ where $H_N:=\text{span}\set{\vphi_k}_{k=1}^N$ and $\vphi_k$ are defined in Section \ref{sec:Pre}.  We denote $\frac{d\tau}{dt}$ by $\dot{\tau}$.  Applying $A^{r/2}e^{\tau A^{1/2}}$ to both sides of \eqref{Gal_EV_mo} and using the identity
\begin{equation*}
 e^{\tau A^{1/2}}\frac{d}{dt}u_N
=
\frac{d}{dt}(e^{\tau A^{1/2}}u_N)-\dot{\tau}A^{1/2}e^{\tau A^{1/2}}u_N,
\end{equation*}
we have
\begin{equation*}
\begin{split}
&\phantom{={}} \alpha^2\pnt{\frac{d}{dt}(A^{r/2}e^{\tau A^{1/2}}A u_N)
-\dot{\tau}A^{(r+1)/2}e^{\tau A^{1/2}}A u_N }
\\&\phantom{={}}+
\frac{d}{dt}(A^{r/2}e^{\tau A^{1/2}}u_N)
-\dot{\tau}A^{(r+1)/2}e^{\tau A^{1/2}}u_N
=
-A^{r/2}e^{\tau A^{1/2}}P_NB(u_N,u_N).
\end{split}
\end{equation*}
Taking the $L^2$ inner product of this equation with $A^{r/2}e^{\tau A^{1/2}}u_N$ and using the facts that $A$ is a positive self-adjoint operator and $\nabla\cdot u_N=0$ gives
\begin{equation}\label{Ge1}
\begin{split}
&\quad\frac{1}{2}\frac{d}{dt}
\pnt{\alpha^2\|A^{r/2}e^{\tau A^{1/2}}\nabla u_N\|^2_{L^2({\nT^3})}
+\|A^{r/2}e^{\tau A^{1/2}} u_N\|^2_{L^2({\nT^3})}}
\\&=
\alpha^2\dot{\tau}\|A^{r/2+1/4}e^{\tau A^{1/2}}\nabla u_N\|^2_{L^2({\nT^3})}
+
\dot{\tau}\|A^{r/2+1/4}e^{\tau A^{1/2}} u_N\|^2_{L^2({\nT^3})}
\\
&\quad
-\pnt{A^{r/2}e^{\tau A^{1/2}}B(u_N,u_N),A^{r/2}e^{\tau A^{1/2}}u_N}.
\end{split}
\end{equation}
We now estimate the last term in \eqref{Ge1}.   We use techniques similar to \cite{Levermore_Oliver1997, Oliver_Titi2000}.  Let us set $\vphi=\psi=u_N$ in Lemma \ref{lemma:LivermoreOliver} so that we have
\begin{align*}
&\phantom{={}}\abs{\pair{A^{r/2}e^{\tau A^{1/2}}B(u_N, u_N)}{A^{r/2}e^{\tau A^{1/2}} u_N}}
\\&\leq
c_1\|A^{r/2} u_N\|^3_{L^2}
+\tau c_2\|A^{r/2} e^{\tau A^{1/2}}u_N\|_{L^2}\|A^{r/2+1/4}
e^{\tau A^{1/2}}u_N\|^2_{L^2}.
\end{align*}
Using this in \eqref{Ge1} gives
\begin{equation}\label{Ge2}\begin{split}
&\quad\frac{1}{2}\frac{d}{dt}
\pnt{\alpha^2\|A^{r/2}e^{\tau A^{1/2}}\nabla u_N\|^2_{L^2({\nT^3})}
+
\|A^{r/2}e^{\tau A^{1/2}} u_N\|^2_{L^2({\nT^3})}}
\\&\leq
c_1\|A^{r/2} u_N\|^3_{L^2({\nT^3})}
+
\dot{\tau}\alpha^2\|A^{r/2+1/4}e^{\tau A^{1/2}}\nabla u_N\|^2_{L^2({\nT^3})}
\\&\quad
+
\pnt{
\dot{\tau}
+
\tau c_2\|A^{r/2} e^{\tau A^{1/2}}u_N\|_{L^2({\nT^3})}
}
\|A^{r/2+1/4} e^{\tau A^{1/2}}u_N\|^2_{L^2({\nT^3})}.
\end{split}\end{equation}
Note that from the proof of Theorem \ref{thm:regularity}, we have $\|A^{r/2} u_N(t)\|_{L^2({\nT^3})}\leq C_\alpha(1+|t|)^{p(r)}:=\kappa(t)$, where $p$ is defined in the statement of the theorem.  Recall that we defined $\tau(t) := \sigma\exp\pnt{-c_2\int_0^t h(\xi)\,d\xi}$.  Note that $\tau(0)=\sigma$.   The function $h(t)>0$ is defined by
\begin{align}\label{h_def}
\begin{split}
(h(t))^2&:=1+\alpha^2\|A^{r/2}e^{\sigma A^{1/2}}\nabla u^{in}\|^2_{L^2({\nT^3})}
\\&\phantom{:=1+}+
\|A^{r/2}e^{\sigma A^{1/2}} u^{in}\|^2_{L^2({\nT^3})}
+
2c_1\int_0^{|t|}\kappa^3(\theta)\,d\theta.
\end{split}
\end{align}
This choice of $\tau$ follows ideas from \cite{Oliver_Titi2000,Levermore_Oliver1997}.  Now, since $\tau,h> 0$ on $[0,\infty)$, we have $\dot{\tau}(t)=-c_2h(t)\tau(t)< 0$.  Furthermore, we claim that
\begin{equation}\label{Atauh}
 \|A^{r/2}e^{\tau(t) A^{1/2}} u_N\|_{L^2({\nT^3})}< h(t)
\end{equation}
 for all time $t\geq0$.  Clearly \eqref{Atauh} holds at $t=0$, and therefore by continuity it holds for a short time.  Let
$t^*:=\sup\set{\theta>0\text{ such that \eqref{Atauh} holds on } [0,\theta)}$.  If $t^*=\infty$, we are done.  Thus, suppose that $t^*<\infty$.  Using the above bounds, integrating \eqref{Ge2} on $[0,t^*]$, and recalling that $\tau(0)=\sigma$, we have
\begin{equation*}\begin{split}
&\;\quad
\alpha^2\|A^{r/2}e^{\tau(t^*) A^{1/2}}\nabla u_N(t^*)\|^2_{L^2({\nT^3})}
+
\|A^{r/2}e^{\tau(t^*) A^{1/2}} u_N(t^*)\|^2_{L^2({\nT^3})}
\\&\leq
\alpha^2\|A^{r/2}e^{\sigma A^{1/2}}\nabla P_Nu^{in}\|^2_{L^2({\nT^3})}
+
\|A^{r/2}e^{\sigma A^{1/2}} P_Nu^{in}\|^2_{L^2({\nT^3})}
+
2c_1\int_0^{t^*}\kappa^3(\theta)\,d\theta
\\&< h^2(t^*).
\end{split}\end{equation*}
Therefore the (strict) inequality \eqref{Atauh} holds on $[0,t^*]$, and can thus be extended in time beyond $t^*$, contradicting the definition of $t^*$.  Hence, the assumption that $t^*<\infty$ must be false.  Thus for all time,
$\dot{\tau}
+
\tau c_2\|A^{r/2} e^{\tau A^{1/2}}u_N\|_{L^2({\nT^3})}
\leq
\dot{\tau}
+
\tau c_2 h
=0
$.
Combining the above observations with \eqref{Ge2} yields
\begin{equation}\label{GevBnd1}
\frac{1}{2}\frac{d}{dt}
\pnt{\alpha^2\|A^{r/2}e^{\tau A^{1/2}}\nabla u_N\|^2_{L^2({\nT^3})}
+
\|A^{r/2}e^{\tau A^{1/2}} u_N\|^2_{L^2({\nT^3})}}
\leq
c_1\kappa^3(t).
\end{equation}
Integrating \eqref{GevBnd1} and using the fact that $\tau(0)=\sigma$, we have
\begin{align}
&\phantom{{}=}\label{Gev_Bnd_Int}
\alpha^2\|A^{r/2}e^{\tau(t) A^{1/2}}\nabla u_N(t)\|^2_{L^2({\nT^3})}
+
\|A^{r/2}e^{\tau(t) A^{1/2}} u_N(t)\|^2_{L^2({\nT^3})}
\\&\leq\nonumber
\alpha^2\|A^{r/2}e^{\sigma A^{1/2}}\nabla P_Nu^{in}\|^2_{L^2({\nT^3})}
+
\|A^{r/2}e^{\sigma A^{1/2}} P_Nu^{in}\|^2_{L^2({\nT^3})}
+
2c_1\int_0^t\kappa^3(\xi)\,d\xi.
\end{align}
   Set $R:=\alpha^2\|A^{r/2}e^{\sigma A^{1/2}}\nabla u^{in}\|^2_{L^2({\nT^3})}+  \|A^{r/2}e^{\sigma A^{1/2}}u^{in}\|_{L^2(\nT^3)}$ and fix $\widetilde{T}\in(0,T]$ arbitrarily.  Define  $\tau_{\min}:=\min_{0\leq t\leq \widetilde{T}}\tau(t)=\tau(\widetilde{T})$ (since $\tau$ is decreasing).  Notice that $\tau_{\min}$ depends on $r$, $\sigma$, $\widetilde{T}$, and $R$ alone.  Consider the Banach  space $B_{r,\sigma,\widetilde{T},R}^\alpha$ (which is in fact a Hilbert space) given by
   \begin{align*}
      B_{r,\sigma,\widetilde{T},R}^\alpha := \set{\vphi\in H : \alpha^2\|A^{r/2}e^{\tau_{\min} A^{1/2}}\nabla\vphi\|_{L^2(\nT^3)}^2+\|A^{r/2}e^{\tau_{\min} A^{1/2}}\vphi\|_{L^2(\nT^3)}^2<\infty},
   \end{align*}
   endowed with the indicated norm.  Due to \eqref{Gev_Bnd_Int} and the fact that $\tau(0)=\sigma$, we have
\begin{align}
&\phantom{{}=}\nonumber
\alpha^2\|A^{r/2}e^{\tau_{\min} A^{1/2}}\nabla u_N(t)\|^2_{L^2({\nT^3})}
+
\|A^{r/2}e^{\tau_{\min} A^{1/2}} u_N(t)\|^2_{L^2({\nT^3})}
\\&\leq\nonumber
\sup_{0\leq t\leq \widetilde{T}}\pnt{
\alpha^2\|A^{r/2}e^{\tau(t) A^{1/2}}\nabla u_N(t)\|^2_{L^2({\nT^3})}
+
\|A^{r/2}e^{\tau(t) A^{1/2}} u_N(t)\|^2_{L^2({\nT^3})}
}
\\&\leq\nonumber
\alpha^2\|A^{r/2}e^{\sigma A^{1/2}}\nabla P_N u^{in}\|^2_{L^2({\nT^3})}
+
\|A^{r/2}e^{\sigma A^{1/2}} P_N u^{in}\|^2_{L^2({\nT^3})}
+
2c_1\int_0^{\widetilde{T}}\kappa^3(\xi)\,d\xi
\\&\leq\nonumber
\alpha^2\|A^{r/2}e^{\sigma A^{1/2}}\nabla u^{in}\|^2_{L^2({\nT^3})}
+
\|A^{r/2}e^{\sigma A^{1/2}} u^{in}\|^2_{L^2({\nT^3})}
+
2c_1\int_0^{\widetilde{T}}\kappa^3(\xi)\,d\xi.
\end{align}
 Thus, $u_N$ is uniformly bounded with respect to $N$ in $C([0,\widetilde{T}],B_{r,\sigma,\widetilde{T},R}^\alpha)$.  We now find a uniform bound for the time derivative.  To do this, we use the fact that $\|uv\|_{B_{s,\sigma,\widetilde{T},R}^\alpha}\leq C_\alpha\|u\|_{B_{s,\sigma,\widetilde{T},R}^\alpha}\|v\|_{B_{s,\sigma,\widetilde{T},R}^\alpha}$ for $s>1/2$, which was  shown in Lemma 1 of \cite{Ferrari_Titi1998}.   Applying $(I+\alpha^2 A)^{-1}$ to \eqref{Gal_EV_mo}, we calculate
{\allowdisplaybreaks
\begin{align*}
   &\phantom{{}=}
   \norm{\frac{d}{dt}u_N}_{B_{r+1,\sigma,\widetilde{T},R}^\alpha}
   =\norm{(I+\alpha^2 A)^{-1}P_NB(u_N,u_N)}_{B_{r+1,\sigma,\widetilde{T},R}^\alpha}
   \\&\leq
   C_\alpha\norm{B(u_N,u_N)}_{B_{r-1,\sigma,\widetilde{T},R}^\alpha}
   \leq
   C_\alpha\norm{u_N}_{B_{r-1,\sigma,\widetilde{T},R}^\alpha}
   \norm{\nabla u_N}_{B_{r-1,\sigma,\widetilde{T},R}^\alpha}
   \\&\leq
   C_\alpha\norm{u_N}_{B_{r-1,\sigma,\widetilde{T},R}^\alpha}
   \norm{u_N}_{B_{r,\sigma,\widetilde{T},R}^\alpha}
   \leq
   C_\alpha \norm{u_N}_{B_{r,\sigma,\widetilde{T},R}^\alpha}^2.
\end{align*}
}
 Thus, due to the uniform boundedness of the sequence $\norm{u_N}_{B_{r,\sigma,\widetilde{T},R}^\alpha}$, we have that $\frac{d}{dt}u_N$ is uniformly bounded in $C([0,\widetilde{T}],B_{r+1,\sigma,\widetilde{T},R}^\alpha)$ with respect to $N$, so that $\set{u_N}_{N=1}^\infty$ is equicontinuous with values in $B_{r,\sigma,\widetilde{T},R}^\alpha$.  Let $0<\epsilon\ll 1$.  Thanks to the compact embedding $B_{r,\sigma,\widetilde{T},R}^\alpha\hookrightarrow B_{r-\epsilon,\sigma,\widetilde{T},R}^\alpha$ and the Arzel\`a-Ascoli Theorem, we see by the above bounds that there exists a subsequence of $\set{u_N}_{N=1}^\infty$, which converges in $C([0,\widetilde{T}],B_{r-\epsilon,\sigma,\widetilde{T},R}^\alpha)$ to an element $w\in C([0,\widetilde{T}],B_{r-\epsilon,\sigma,\widetilde{T},R}^\alpha)$.  As shown in Proof 2 of Theorem \ref{thm:regularity}, this sequence also converges in $C([0,\widetilde{T}],V)$ to the unique solution $u$ given in Theorem \ref{thm:regularity}.  By the uniqueness of limits, $u=w\in C([0,\widetilde{T}],B_{r-\epsilon,\sigma,\widetilde{T},R}^\alpha)$, and furthermore, we have the bound
 \begin{align}
&\phantom{{}=}\nonumber
\sup_{t\in[0,\widetilde{T}]}\|u(t)\|_{B_{r-\epsilon,\sigma,\widetilde{T},R}^\alpha}^2
\\&\leq\nonumber
\alpha^2\|A^{r/2}e^{\sigma A^{1/2}}\nabla u^{in}\|^2_{L^2({\nT^3})}
+
\|A^{r/2}e^{\sigma A^{1/2}} u^{in}\|^2_{L^2({\nT^3})}
+
2c_1\int_0^{\widetilde{T}}\kappa^3(\xi)\,d\xi.
\end{align}
Since this bound holds uniformly for all $0<\epsilon\ll 1$, it also holds for $\epsilon=0$, which can be seen by using the Fourier series representation
\begin{align*}
   \|u(t)\|_{B_{r-\epsilon,\sigma,\widetilde{T},R}^\alpha}^2
   =\sum_{j\in\nZ^3}(1+\alpha^2|j|^2)|j|^{2(r-\epsilon)}e^{2\tau_{\text{min}}|j|}|\hat{u}_j(t)|^2,
\end{align*}
taking the limit as $\epsilon\maps0$ and passing the limit inside the infinite sum using, for example,  the Lebesgue Monotone Convergence Theorem for the counting measure.
Here, we have used the notation $\hat{u}_j(t)$ for the Fourier coefficients of $u(t)$.  Thus,
\begin{align*}
&\phantom{{}=}\nonumber
\alpha^2\|A^{r/2}e^{\tau_{\min} A^{1/2}}\nabla u(t)\|^2_{L^2({\nT^3})}
+
\|A^{r/2}e^{\tau_{\min} A^{1/2}} u(t)\|^2_{L^2({\nT^3})}
\\&\leq\nonumber
\alpha^2\|A^{r/2}e^{\sigma A^{1/2}}\nabla u^{in}\|^2_{L^2({\nT^3})}
+
\|A^{r/2}e^{\sigma A^{1/2}} u^{in}\|^2_{L^2({\nT^3})}
+
2c_1\int_0^{\widetilde{T}}\kappa^3(\xi)\,d\xi,
\end{align*}
for all $t\in[0,\widetilde{T}]$.  Finally, by choosing $t=\widetilde{T}$ in the above inequality, and by   recalling that $\tau_{\min}=\tau(\widetilde{T})$ and that $\widetilde{T}$ was chosen arbitrarily in $(0,T]$, we have established \eqref{Gev_Bnd_Thm} for $u$, with $\widetilde{T}$ playing the role of $t$.
\end{proof}

\section{Convergence to the 3D Euler Equations and a Blow-up Criterion}\label{sec:Blowup}

We now consider the three-dimensional Euler equations, namely \eqref{NSE} with $\nu=0$ and periodic boundary conditions, or equivalently \eqref{fun_Euler}.  It is known that if $u^{in}\in H^s(\nT^3)\cap V$ for $s>5/2$, then there exists a unique solution to these equations in $C([0,T_*],H^s(\nT^3)\cap V))\cap C^1([0,T_*],H^{s-1}(\nT^3)\cap V))$ for some $T_*>0$ (see, e.g., \cite{Lichtenstein_1930, Kato1972, Kato1974, Temam1976, Marchioro_Pulvirenti_1994, Majda_Bertozzi_2002}). Observe, however, that the question of existence of weak solutions for the Euler equations, i.e., for $u^{in}\in H^s$, $s<5/2$, or $u^{in}\in C^\delta$, $\delta\in[0,1]$ is still open.  In fact, the only results for short-time existence are those for $u^{in}\in C^{1,\delta}$, proven in \cite{Lichtenstein_1930}, or $u^{in}\in H^s$, $s>5/2$, (see, e.g., \cite{Marchioro_Pulvirenti_1994,  Majda_Bertozzi_2002, Kato1972}), where the solutions are well-posed in the sense of Hadamard.  Recently, it was shown in \cite{Bardos_Titi_2009} that the three-dimensional Euler equations are ill-posed in the space $C^{\delta}$.

For arbitrary $u^{in}\in H^s(\nT^n)\cap V$, with $s\geq 3$, we consider the maximal time interval $[0,T_{\text{max}})$, $T_{\text{max}}\geq T_*$, for which a solution to the Euler equations exists and is unique.  It is a major open problem in mathematics to determine whether $T_{\text{max}}$ is finite; that is, whether solutions exist globally in time or experience blow-up in finite time (see the recent surveys \cite{Bardos_Titi_2007} and \cite{Constantin2007}).  In the present section, we present a criterion for blow-up of the Euler equations using methods similar to those in \cite{Khouider_Titi2008}. There are few other known criteria for blow-up of the Euler equations which are directly checkable for the Euler equations themselves (see, e.g., \cite{Beale_Kato_Majda_1984, Ferrari_1993, Constantin_Fefferman_Majda_1996, Deng_Hou_Yu_2005}).

\begin{theorem}[Convergence]\label{thm:convergence}
    Given initial data $u^{in}, u_\alpha^{in}\in H^s(\nT^3)$ for some $s\geq3$, let $u, u_\alpha \in C([0,T],H^s({\nT^3}))\cap C^1([0,T],H^{s-1}({\nT^3}))$ be the corresponding solutions to \eqref{NSE} and \eqref{EV}, respectively, where $0<T<T_{\text{max}}$, and $T_{\text{max}}$ is the maximal time for which a solution to the Euler equations exists and is unique.  Then:
 \begin{itemize}
    \item[(i)] For all $t\in[0,T]$,
    \begin{align}
\nonumber&
\phantom{{}=}\|u(t)-u_\alpha(t)\|_{L^2({\nT^3})}^2
+\alpha^2\|\nabla(u(t)- u_\alpha(t))\|_{L^2({\nT^3})}^2
\\\label{conv_claim}&\leq
(\|u^{in}-u_\alpha^{in}\|_{L^2({\nT^3})}^2
+\alpha^2\|\nabla(u^{in}- u_\alpha^{in})\|_{L^2({\nT^3})}^2)e^{Ct}
+C\alpha^2(e^{Ct}-1).
\end{align}

    \item[(ii)] Consequently, if
    \begin{equation}\label{alpha_conv}
 \|u^{in}-u_\alpha^{in}\|_{L^2({\nT^3})}^2
+\alpha^2\|\nabla(u^{in}- u_\alpha^{in})\|_{L^2({\nT^3})}^2\maps 0\text{ as } \alpha\maps0
\end{equation}
 (in particular, when $u_\alpha^{in}= u^{in}$ for all $\alpha>0$), then $u_\alpha\maps u$ in  $C([0,T],H)$, as $\alpha\maps0$.
 \end{itemize}
\end{theorem}

\begin{remark} Regarding part (\textit{ii}) of Theorem \ref{thm:convergence}, one may find many sufficient conditions for \eqref{alpha_conv} to hold.  For example, if $u^{in}_\alpha\maps u^{in}$ in $L^2(\nT^3)$, as $\alpha\maps0$, and either $\alpha\|u^{in}_\alpha\|_{H^1(\nT^3)}\maps0$, as $\alpha\maps0$, or $\|u^{in}_\alpha\|_{H^\sigma(\nT^3)}\leq C/\alpha^\sigma$ for some $\sigma\in(1,s]$, then \eqref{alpha_conv} holds, due to the interpolation inequality (see, e.g., \cite{Adams2003})
\[
   \|u_\alpha^{in}-u^{in}\|_{H^1(\nT^3)}\leq C\|u_\alpha^{in}-u^{in}\|_{L^2(\nT^3)}^{1-1/\sigma}\|u_\alpha^{in}-u^{in}\|_{H^\sigma(\nT^3)}^{1/\sigma}.
\]
Therefore $u_\alpha\maps u$ in $C([0,T],H)$, as $\alpha\maps0$.
\end{remark}

\begin{remark}  With slightly more effort, one can prove
similar convergence results to Theorem \ref{thm:convergence} for
$u^{in}, u^{in}_\alpha\in H^s(\nT^3)$ with $s>5/2$.
However, for simplicity of presentation, we take
$s\geq3$.  In fact, in lieu of the result in \cite{Pak_Park_2004},
which shows the local well-posedness of the Euler equations in
the Besov space $B_{\infty,1}^1(\nT^n)$, $n\geq 2$, one can
prove a similar result to Theorem \ref{thm:convergence} assuming only that
$u^{in}, u^{in}_\alpha\in B_{\infty,1}^1(\nT^n)$.
\end{remark}

\begin{proof}
Subtracting \eqref{EV} from the Euler equations gives
\begin{align*}
&
\partial_t (u-u_\alpha)
+\alpha^2\partial_t\triangle u_\alpha +\nabla (p-p_\alpha)
\\=&
-(u\cdot\nabla) u+(u_\alpha\cdot\nabla) u_\alpha
\\=&
-((u-u_\alpha)\cdot\nabla)u
+((u-u_\alpha)\cdot\nabla)(u-u_\alpha)
-(u\cdot\nabla)(u-u_\alpha).
\end{align*}
We note that since $u, u_\alpha\in C^1([0,T], H^{s-1})$ and $s\geq3$, we have
$u_t$, $\partial_t u_\alpha$, $\triangle\partial_t u$,
$\triangle\partial_t u_\alpha\in C([0,T],L^2(\nT^n))$, so that
taking inner products with these terms is justified.  This is
because $H^{s-1}(\nT^3)$ is an algebra for $s\geq3$ (see, e.g.,
\cite{Adams2003}), so we also have $u\cdot\nabla u,
(u_\alpha\cdot\nabla) u_\alpha \in C([0,T],H^{s-1}(\nT^n))$.  Using
\eqref{B:zero} and the fact that $\nabla\cdot u=\nabla\cdot
u_\alpha=0$, we integrate against $u-u_\alpha$ to find
\begin{equation*}
\frac{1}{2}\frac{d}{dt}\|u-u_\alpha\|_{L^2({\nT^3})}^2
+\alpha^2\int_{\nT^3}\triangle\partial_t
u_\alpha\cdot(u-u_\alpha)\,dx =-\int_{\nT^3}
((u-u_\alpha)\cdot\nabla)u\cdot(u-u_\alpha)\,dx.
\end{equation*}
The exchange of the order of the integral and the time derivative is justified here and below since $u,u_\alpha$ are $C^1$ in time with values in $H^{s-1}$.  Adding and subtracting the term $\alpha^2\int_{\nT^3}(\triangle u_t)\cdot(u-u_\alpha)\,dx$, we have
\begin{align*}
&\frac{1}{2}\frac{d}{dt}\|u-u_\alpha\|_{L^2({\nT^3})}^2
-\alpha^2\int_{\nT^3}(\triangle u_t-\triangle
\partial_t u_\alpha)\cdot(u-u_\alpha)\,dx
\\&+\alpha^2\int_{\nT^3}(\triangle u_t)\cdot(u-u_\alpha)\,dx
=
-\int_{\nT^3} ((u-u_\alpha)\cdot\nabla)u\cdot(u-u_\alpha)\,dx.
\end{align*}
Integrating by parts and rearranging gives
\begin{align}\label{uAlpha}
\nonumber&\phantom{={}}\frac{1}{2}\frac{d}{dt}\pnt{\|u-u_\alpha\|_{L^2({\nT^3})}^2
+\alpha^2\|\nabla(u- u_\alpha)\|_{L^2({\nT^3})}^2}
\\&=
-\alpha^2\int_{\nT^3}(\triangle u_t)(u-u_\alpha)\,dx
-\int_{\nT^3} ((u-u_\alpha)\cdot\nabla)u\cdot(u-u_\alpha)\,dx.
\end{align}
Using the fact that $u$ satisfies the Euler equations, we have
\begin{align*}
\nonumber-\alpha^2\int_{\nT^3}(\triangle u_t)\cdot(u-u_\alpha)\,dx
&=
\alpha^2\int_{\nT^3}(- u_t)\cdot\triangle(u-u_\alpha)\,dx
\\
\begin{split}
&=\alpha^2\int_{\nT^3}((u\cdot\nabla)u)\cdot\triangle(u-u_\alpha)\,dx
\\&\phantom{={}}
+\alpha^2\int_{\nT^3}\nabla p\cdot\triangle(u-u_\alpha)\,dx
\end{split}
\\\nonumber&=
-\alpha^2\int_{\nT^3}\nabla ((u\cdot\nabla)u)\cdot
\nabla(u-u_\alpha)\,dx.
\end{align*}
The last equality follows by using integration by parts on the first integral and noting that the second integral term is zero because of \eqref{deRham} and the fact that $\text{div}(\triangle(u-u_\alpha))=0$.  Thus the right-hand side of \eqref{uAlpha} becomes
$$-\alpha^2\int_{\nT^3}(\nabla ((u\cdot\nabla)u))\nabla(u-u_\alpha)\,dx
+\int_{\nT^3} ((u-u_\alpha)\cdot\nabla)u\cdot(u-u_\alpha)\,dx.
$$
Since we are assuming $u$ is a regular solution, i.e., $u\in H^s$ for some $s\geq3$, the $H^3$ norm is finite, so the second term is bounded above by
\begin{align*}
\|\nabla u\|_{L^\infty({\nT^3})}
\|u- u_\alpha\|^2_{L^2({\nT^3})}
\leq
C\|u\|_{H^{3}({\nT^3})}
\|u- u_\alpha\|^2_{L^2({\nT^3})}
\leq
\tilde{C}\|u- u_\alpha\|^2_{L^2({\nT^3})}.
\end{align*}
Using the Cauchy-Schwarz inequality on the first term gives
\begin{align*}
&\phantom{={}}
\abs{-\alpha^2\int_{\nT^3}(\nabla ((u\cdot\nabla)u))\cdot\nabla(u-u_\alpha)\,dx}
\\&\leq
\alpha^2\|\nabla ((u\cdot\nabla)u)\|_{L^2({\nT^3})}\|\nabla(u-u_\alpha)\|_{L^2({\nT^3})}
\\&\leq
\alpha^2\pnt{\||\nabla u|^2\|_{L^2({\nT^3})}+\|D^2u\|_{L^2({\nT^3})}}
\|\nabla(u-u_\alpha)\|_{L^2}
\\&\leq
\alpha^2\pnt{\|\nabla u\|_{L^4({\nT^3})}+\|u\|_{H^2({\nT^3})}}
\|\nabla(u-u_\alpha)\|_{L^2}
\\&\leq
\alpha^2\pnt{C\|\nabla u\|_{H^2({\nT^3})}+\|u\|_{H^2({\nT^3})}}
\|\nabla(u-u_\alpha)\|_{L^2}
\\&\leq
2C\alpha^2
\|\nabla(u-u_\alpha)\|_{L^2}
\\&\leq
C^2\alpha^2+\alpha^2
\|\nabla(u-u_\alpha)\|^2_{L^2}.
\end{align*}
  Collecting the above estimates we have from \eqref{uAlpha}
 \begin{align}
\nonumber&
\phantom{{}=}\frac{d}{dt}\pnt{\|u-u_\alpha\|_{L^2({\nT^3})}^2
+\alpha^2\|\nabla(u- u_\alpha)\|_{L^2({\nT^3})}^2}
\\&\leq
C_1\pnt{\|u- u_\alpha\|^2_{L^2({\nT^3})}
+
\alpha^2\|\nabla(u-u_\alpha)\|^2_{L^2({\nT^3})}}
+
C_2\alpha^2.
\end{align}
Gr\"onwall's inequality then gives
 \begin{align}
\nonumber&
\phantom{{}=}\|u(t)-u_\alpha(t)\|_{L^2({\nT^3})}^2
+\alpha^2\|\nabla(u(t)- u_\alpha(t))\|_{L^2({\nT^3})}^2
\\\label{ModEnergyBound}&\leq
(\|u^{in}-u_\alpha^{in}\|_{L^2({\nT^3})}^2
+\alpha^2\|\nabla(u^{in}- u_\alpha^{in})\|_{L^2({\nT^3})}^2)e^{C_1t}
+\alpha^2\frac{C_2}{C_1}(e^{C_1t}-1).
\end{align}
This proves \eqref{conv_claim}.  Next, if  $u_\alpha^{in}\maps u^{in}$ in $L^2(\nT^3)$ and $\alpha^2\|\nabla u_\alpha^{in}\|_{L^2(\nT^3)}^2\leq M<\infty$ is bounded as $\alpha\maps 0$, noting that the left-hand side of \eqref{conv_claim} is bounded below by $\|u-u_\alpha\|_{L^2({\nT^3})}^2$, we take the $\limsup_{\alpha\maps0}$ to find that $u_\alpha\maps u$ in $C([0,T),H)$.
\end{proof}
Due to Theorem \ref{thm:convergence}, we have the following result.
\begin{theorem}[Blow-up Criterion]
Assume $u^{in}\in H^s$, for some $s\geq3$. Suppose there exists a finite time $T^{**} >0$ such that the solutions $u_\alpha$ of \eqref{fun_Euler_Voigt} with $u_\alpha^{in}=u^{in}$ for each $\alpha>0$ satisfy
\[
   \sup_{t\in[0,T^{**})}\limsup_{\alpha\maps0}\alpha^2\|\nabla u_\alpha(t)\|^2_{L^2({\nT^3})}>0.
\]
Then the Euler equations with initial data $u^{in}$ develop a singularity in the interval $[0,T^{**}]$.
\end{theorem}
\begin{proof}
 Suppose to the contrary that the solution $u$ to the Euler equations is regular on $[0,T^{**}]$.  By \eqref{H1} we have the ``modified energy'' equality
\begin{equation}\label{EnEq}
\|u_\alpha(t)\|_{L^2({\nT^3})}^2
+\alpha^2\|\nabla u_\alpha(t)\|_{L^2({\nT^3})}^2
=
\|u^{in}\|_{L^2({\nT^3})}^2
+\alpha^2\|\nabla u^{in}\|_{L^2({\nT^3})}^2,\end{equation}
for all $t\geq0$.  Taking the limsup of \eqref{EnEq} gives
\begin{equation}\label{limsupEq}
\|u(t)\|_{L^2({\nT^3})}^2
+\limsup_{\alpha\maps0}\alpha^2\|\nabla u_\alpha(t)\|_{L^2({\nT^3})}^2
=
\|u^{in}\|_{L^2({\nT^3})}^2,
\end{equation}
for all $t\in[0,T^{**}]$, since $\|u_\alpha(t)\|_{L^2}\maps\|u(t)\|_{L^2}$ for $t\in[0,T^{**}]$ by \eqref{ModEnergyBound}.  Since regular solutions of the Euler equations conserve energy, we have the energy equality $\|u(t)\|_{L^2({\nT^3})}=\|u^{in}\|_{L^2({\nT^3})}$ for $t\in[0,T^{**}]$.  Thus \eqref{limsupEq} contradicts the hypothesis and $u$ must blow-up in $[0,T^{**}]$.
\end{proof}

\section{The MHD-Voigt Case}\label{sec:MHD}
In this section, we consider a Voigt-type regularization of the inviscid, irressitive magneto-hydrodynamic (MHD) equations, given by
\begin{subequations}\label{IMHDV}
\begin{align}
 -\alpha^2\partial_t \triangle u+\partial_tu +(u\cdot\nabla)u
+\nabla p+\frac{1}{2}\nabla|\B|^2&=(\B\cdot\nabla)\B,
\label{IMHDV1}\\
 -\alpha_M^2\partial_t \triangle \B+\partial_t\B +(u\cdot\nabla)\B-(\B\cdot\nabla)u+\nabla q&=0,
\label{IMHDV2}\\
 \nabla\cdot \B = \nabla\cdot u &= 0,\label{IMHDV3}\\
 \B(0)=\B^{in},\; u(0)&=u^{in},\label{IMHDV4}
\end{align}
\end{subequations}
where $\alpha,\alpha_M>0$, and the boundary conditions are taken to be periodic, and we also assume as before that
\[
  \int_{\nT^3} u\,dx =  \int_{\nT^3} \B\,dx =0.
\]
Here, the unknowns are the fluid velocity field $u(x,t)$, the fluid pressure $p(x,t)$, the magnetic pressure $q(x,t)$, and the magnetic field $\B(x,t)$.  (By formally taking the divergence of \eqref{IMHDV2}, we find that $\nabla q\equiv0$, but this is not assumed \textit{a priori}.) Note that when $\alpha=\alpha_M=0$, we formally retrieve the inviscid, irressitive MHD equations, i.e., the case where the fluid viscosity and the magnetic diffusivity are equal to zero.  The viscous Bardina model for the MHD case has been studied in \cite{LabovschiiTrenchea2009}, and another viscous MHD $\alpha$-model has been studied in \cite{Linshiz_Titi_2007}.  Here, we study the inviscid case, and hence prove stronger results than those reported in \cite{LabovschiiTrenchea2009}.

We first prove that this system has a unique local (in time) solution.  Let us show that \eqref{IMHDV} has a unique short-time solution.  To do this, we follow \cite{Cao_Lunasin_Titi2006} and use the Picard-Lindel\"of Theorem.  Applying $P_\sigma$ (see Section \ref{sec:Pre}) to \eqref{IMHDV}, we obtain the equivalent system
\begin{subequations}\label{LHproj}
\begin{align}
 \pd{}{t}\pnt{\alpha^2 A u+u} &=B(\B,\B)-B(u,u),
\label{LHproj1}\\
 \pd{}{t}\pnt{\alpha_M^2 A \B+\B}&=B(\B,u)-B(u,\B),
\label{LHproj2}\\
 \B(0)&=\B^{in},\; u(0)=u^{in}.\label{LHproj3}
\end{align}
\end{subequations}
Note that it is possible to recover $p$ and $q$ by using \eqref{deRham} (see, e.g., \cite{Lions_1959}).
Denote $v=(\alpha^2 A +I)u$, $Z=(\alpha_M^2 A +I)\B$, $N_1(v,Z)=B(\B,\B)-B(u,u)$, and $N_2(v,Z)=B(\B,u)-B(u,\B)$. Then \ref{LHproj} is equivalent to the system
\begin{equation}\label{vectsys}
 \pd{}{t} \binom{v}{Z}
=\binom{N_1(v,Z)}{N_2(v,Z)}
,\:\:
\binom{v(0)}{Z(0)}=\binom{v^{in}}{Z^{in}}
:=\binom{-\alpha^2 \triangle u^{in} +u^{in}}{-\alpha_M^2 \triangle\B^{in} +\B^{in}}.
\end{equation}
Using this form of \eqref{IMHDV}, we will now prove the following theorem.

\begin{theorem}[Short-time existence and uniqueness]\label{MHDshort}
Let $v^{in},Z^{in}\in V'$.\\ Then there exists a time $T
=T(\|v^{in}\|_{V'},\|Z^{in}\|_{V'})>0$ such that \eqref{vectsys} has a unique solution $(v,Z)\in C^1([-T,T],V')$, or equivalently, $(u,\B)\in C^1([-T,T],V)$.
\end{theorem}
\begin{proof}
We follow almost word by word the last section of \cite{Cao_Lunasin_Titi2006}.  To show local existence, it is enough to show that $N_1$ and $N_2$ are locally Lipschitz in the space $V'$.  By the Poincar\'e inequality \eqref{poincare} and \eqref{LHproj1} we have
{\allowdisplaybreaks
\begin{align*}
 &\phantom{={}}\|N_1(v_1,Z_1)-N_1(v_2,Z_2)\|_{V'}
\\ &=
\|B(\B_1,\B_1)-B(u_1,u_1)-B(\B_2,\B_2)+B(u_2,u_2)\|_{V'}
\\ &=
\|B(u_1-u_2,u_2)-B(u_1,u_1-u_2)+B(\B_1-\B_2,\B_2)-B(\B_1,\B_1-\B_2)\|_{V'}
\\ &\leq
\|B(u_1-u_2,u_2)-B(u_1,u_1-u_2)\|_{V'}\! +\!\|B(\B_1-\B_2,\B_2)-B(\B_1,\B_1-\B_2)\|_{V'}
\\
\begin{split}
&=\sup_{\set{w\in V:\|w\|=1}}
|\ip{B(u_1-u_2,u_2)-B(u_1,u_1-u_2)}{w}|
\\&\phantom{=}+{}\sup_{\set{w\in V:\|w\|=1}}
|\ip{B(\B_1-\B_2,\B_2)-B(\B_1,\B_1-\B_2)}{w}|
\end{split}
\\\begin{split}
&\leq C|u_1-u_2|^{1/2}\|u_1-u_2\|^{1/2}\|u_2\|
+C|u_1|^{1/2}\|u_1\|^{1/2}\|u_1-u_2\|
\\&\phantom{=}
+C|\B_1-\B_2|^{1/2}\|\B_1-\B_2\|^{1/2}\|\B_2\|
+C|\B_1|^{1/2}\|\B_1\|^{1/2}\|\B_1-\B_2\|
\end{split}
\\&\leq
C\lambda_1^{-1/4}\|u_1-u_2\|\pnt{\|u_1\|+\|u_2\|}
+C\lambda_1^{-1/4}\|\B_1-\B_2\|\pnt{\|\B_1\|+\|\B_2\|}
\\&\leq
2CR\lambda_1^{-1/4}\pnt{\|u_1-u_2\|+\|\B_1-\B_2\|}
\\&\leq
CR\lambda_1^{-1/4}\pnt{\|v_1-v_2\|_{V'}+\|Z_1-Z_2\|_{V'}},
\end{align*}
}
where $R$ is chosen so that $\|u_1\|,\|u_2\|,\|\B_1\|,\|\B_2\|<R$.
Next, for $N_2$, we have by the Poincar\'e inequality \eqref{poincare} and \eqref{LHproj1},
{\allowdisplaybreaks
\begin{align*}
 &\phantom{={}}\|N_2(v_1,Z_1)-N_2(v_2,Z_2)\|_{V'}
\\ &=
\|B(\B_1,u_1)-B(u_1,\B_1)+B(u_2,\B_2)+B(\B_2,u_2)\|_{V'}
\\ &\leq
\|B(u_1,\B_1-\B_2)-B(u_1-u_2,\B_2)\|_{V'}
\!+\!\|B(\B_1,u_1-u_2)-B(\B_1-\B_2,u_2)\|_{V'}
\\\begin{split}&=
\sup_{\set{w\in V:\|w\|=1}}
|\ip{B(u_1,\B_1-\B_2)-B(u_1-u_2,\B_2)}{w}|
\\&\phantom{={}}+\sup_{\set{w\in V:\|w\|=1}}
|\ip{B(\B_1,u_1-u_2)-B(\B_1-\B_2,u_2)}{w}|
\end{split}
\\\begin{split}&\leq
C|u_1|^{1/2}\|u_1\|^{1/2}\|\B_1-\B_2\|
+C|u_1-u_2|^{1/2}\|u_1-u_2\|^{1/2}\|\B_2\|
\\&\phantom{={}}
+C|\B_1|^{1/2}\|\B_1\|^{1/2}\|u_1-u_2\|
+C|\B_1-\B_2|^{1/2}\|\B_1-\B_2\|^{1/2}\|u_2\|
\end{split}
\\&\leq
C\lambda_1^{-1/4}(\|u_1\|+\|u_2\|)\|\B_1-\B_2\|
+C\lambda_1^{-1/4}(\|\B_1\|+\|\B_2\|)\|u_1-u_2\|
\\&\leq
2CR\lambda_1^{-1/4}\pnt{\|u_1-u_2\|+\|\B_1-\B_2\|}
\\
\begin{split}
&\leq
CR\lambda_1^{-1/4}\pnt{\|v_1-v_2\|_{V'}+\|Z_1-Z_2\|_{V'}},
\end{split}
\end{align*}
}
where again $R$ is chosen so that $\|u_1\|,\|u_2\|,\|\B_1\|,\|\B_2\|<R$.  Thus, $N_1$ and $N_2$ are locally Lipschitz in $V'$, so the right-hand side of \eqref{vectsys} is as well.  Therefore by the Picard-Lindel\"of Theorem, there exists a unique solution $(v,Z)$ to \eqref{vectsys} such that $v,Z\in C^1([-T,T],V)$ for some $T>0$ which may depend upon the initial conditions.
\end{proof}

We next show that in fact, we have global existence.
\begin{theorem}[Global existence and uniqueness]\label{MHDglobal}
Let $v^{in},Z^{in}\in V'$\! (equivalently $u^{in},\B^{in}\in V$). Then \eqref{vectsys} has a unique solution $(v,Z)\in C^1((-\infty,\infty),V')$ (equivalently $(u,\B)\in C^1((-\infty,\infty),V)$).
\end{theorem}

\begin{proof}
It is sufficient to show that on the maximal time interval of existence $\|v(t)\|_{V'}$ and $\|Z(t)\|_{V'}$ are finite.  Let $[0,T_*)$ be the maximal interval of existence.  If $T_*=\infty$, we are done, so suppose $T_*<\infty$.  Then
we must have
$\limsup_{t\maps T_*}\|v(t)\|_{V'}=\infty$ or $\limsup_{t\maps T_*}\|Z(t)\|_{V'}=\infty$, otherwise we could use Theorem \ref{MHDshort} to extend the solution further in time, contradicting the definition of $T_*$.  Therefore, we must have
\begin{equation}\label{limsups}
\limsup_{t\maps T_*^-}\|u(t)\|=\infty
\tabM\text{or}\tabM
\limsup_{t\maps T_*^-}\|\B(t)\|=\infty.
\end{equation}
Taking the inner product of \eqref{LHproj1} with $u$ and \eqref{LHproj2} with $\mathcal{B}$ and integrating by parts, we have
\begin{subequations}\label{MHDL2ip}
\begin{align}
 \pd{}{t}\pnt{\alpha^2 \|u\|^2+|u|^2} &=(B(\B,\B),u),
\label{MHDL2ip1}\\
 \pd{}{t}\pnt{\alpha_M^2 \|\B\|^2+|\B|^2}&=(B(\B,u),\B)=-(B(\B,\B),u),
\label{MHDL2ip2}
\end{align}
\end{subequations}
where in the last equation, we used \eqref{B:Alt}.  Adding \eqref{MHDL2ip1} and \eqref{MHDL2ip2} and integrating in time, we obtain on $[0,T_*)$
\begin{equation}\label{L2eq}
 \alpha^2 \|u\|^2+\alpha_M^2 \|\B\|^2+|u|^2+|\B|^2
=\alpha^2 \|u^{in}\|^2+\alpha_M^2 \|\B^{in}\|^2+|u^{in}|^2+|\B^{in}|^2.
\end{equation}
Since the right-hand side is finite for all time, this contradicts \eqref{limsups}.  The proof is nearly identical for $(-T_*,0]$.  Thus, $T_*=\infty$.
\end{proof}

The higher-order regularity for system \eqref{IMHDV} holds as well.

\begin{theorem}\label{thm:mag_reg}
   Let $u^{in},\B^{in}\in H^m(\nT^3)\cap V$, for $m\geq 1$.  Then the unique solution $(u,\B)$ of \eqref{IMHDV} given by Theorem \ref{MHDglobal} lies in $C((-\infty,\infty),H^m\cap V)$.
\end{theorem}
\begin{proof}
   The proof follows nearly the same steps as in Section \ref{sec:HReg}, \textit{mutatis mutandis}, so we only sketch the main ideas.  The primary difference lies in handling the additional variables given by the magnetic terms.  We mention that the work here is done formally, but can be made rigorous by following similar ideas to those in either of the two proofs of Theorem \ref{thm:regularity} given above.  One must work inductively, first estimating the $H^m$ norms for $m=0,1,2,3$ as in steps 1-3 of Proof 1 of Theorem \ref{thm:regularity}, and then obtain the estimates for general $m\geq 4$, as in step 4.  The cases $m=0,1$ are already given by \eqref{L2eq}.  Here, we only show the $m=2$ case, where the central issue is more transparent.  The other cases are more complicated notationally, but not conceptually.  We first apply $\partial^\beta$ to \eqref{IMHDV1} for an arbitrary $\beta$ with $|\beta|=1$, and integrate the result against $\partial^\beta u$ to obtain
\begin{align}
\nonumber
 \frac{1}{2}\frac{d}{dt}\pnt{\alpha^2 \|\partial^\beta u\|^2+|\partial^\beta u|^2}
 &=
 (B(\partial^\beta\B,\B),\partial^\beta u)+(B(\B,\partial^\beta\B),\partial^\beta u)
 \\&\phantom{{}=}\nonumber
 -(B(\partial^\beta u,u),\partial^\beta u)-(B(u,\partial^\beta u),\partial^\beta u)
 \\\label{MHD_reg_u}&=
 (B(\partial^\beta\B,\B),\partial^\beta u)+(B(\B,\partial^\beta\B),\partial^\beta u)
 \\&\phantom{{}=}
 -(B(\partial^\beta u,u),\partial^\beta u)
\nonumber
\end{align}
Where we have used \eqref{B:zero}. Next, we apply $\partial^\beta$ to \eqref{IMHDV2} and integrate the result against $\partial^\beta \B$
\begin{align}
\nonumber
 \frac{1}{2}\frac{d}{dt}\pnt{\alpha_M^2 \|\partial^\beta \B\|^2+|\partial^\beta \B|^2}
 &=
 (B(\partial^\beta\B,u),\partial^\beta \B)+(B(\B,\partial^\beta u),\partial^\beta \B)
 \\&\phantom{{}=}\nonumber
 -(B(\partial^\beta u,\B),\partial^\beta \B)-(B(u,\partial^\beta \B),\partial^\beta \B)
 \\\label{MHD_reg_B}&=
 (B(\partial^\beta\B,u),\partial^\beta \B)+(B(\B,\partial^\beta u),\partial^\beta \B)
 \\&\phantom{{}=}\nonumber
 -(B(\partial^\beta u,\B),\partial^\beta \B)
\nonumber
\end{align}
Since $(B(\B,\partial^\beta u),\partial^\beta \B)=-(B(\B,\partial^\beta \B),\partial^\beta u)$ by \eqref{B:Alt}, we may add \eqref{MHD_reg_u} and \eqref{MHD_reg_B}, to obtain a very important cancellation of the terms that involve the highest order derivatives.  This gives us
\begin{align}
\nonumber
&\phantom{{}=}\frac{1}{2}\frac{d}{dt}\pnt{\alpha_M^2 \|\partial^\beta \B\|^2+|\partial^\beta \B|^2
 + \alpha^2 \|\partial^\beta u\|^2+|\partial^\beta u|^2}
 \\\label{MHD_reg_uB_eq} &=
 (B(\partial^\beta\B,u),\partial^\beta \B) -(B(\partial^\beta u,\B),\partial^\beta \B)
 \\&\phantom{{}=}\nonumber
 +(B(\partial^\beta\B,\B),\partial^\beta u) -(B(\partial^\beta u,u),\partial^\beta u)
 \\\nonumber&\leq
  3\|\nabla \B\|_{L^2(\nT^3)}\|\nabla \B\|_{L^3(\nT^3)}\|\nabla u\|_{L^6(\nT^3)}
 \\&\phantom{{}=}\nonumber
 +\|\nabla u\|_{L^2(\nT^3)}\|\nabla u\|_{L^3(\nT^3)}\|\nabla u\|_{L^6(\nT^3)}
\\\nonumber&\leq
C\|\B\|\|\B\|_{H^{3/2}(\nT^3)}\|u\|_{H^2(\nT^3)}
 +C\|u\|\|u\|_{H^{3/2}(\nT^3)}\|u\|_{H^2(\nT^3)}
\\\nonumber&\leq
C_{\alpha,\alpha_M}\pnt{\|\B\|_{H^{3/2}(\nT^3)}^2 +\|u\|_{H^{3/2(\nT^3)}}^2+\|u\|_{H^2(\nT^3)}^2}
 \end{align}
since $\|\B\|$, $\|u\|$ are uniformly bounded from \eqref{L2eq}.  Summing over all $\beta$, with $|\beta|=1$, and using the norm equivalence \eqref{elliptic_reg} yields
\begin{align}
\nonumber
   &\phantom{{}=}\frac{1}{2}\frac{d}{dt}\pnt{\alpha_M^2 \|\B\|_{H^2(\nT^3)}^2
   +\|\B\|^2
 + \alpha^2 \|u\|_{H^2(\nT^3)}^2+\|u\|^2}
 \\\label{MHD_reg_uB_sum}&\leq
 C_{\alpha,\alpha_M}\pnt{\|\B\|_{H^{2}(\nT^3)}^2+\|u\|_{H^2(\nT^3)}^2}.
\end{align}
Thus, the $H^2(\nT^3)$ norms have at most algebraic growth by Gr\"onwall's inequality.  To complete the proof, the estimates for the higher derivatives are carried out in a similar way, using the methods of steps 3 and 4 of Theorem \ref{thm:regularity}, and making use of the analogue of the important cancellation that was used to obtain \eqref{MHD_reg_uB_eq}.  We emphasize again that the estimates here are derived formally, but can be made rigorous, as described above.
\end{proof}

\noindent
\begin{remark}
  One can also obtain results for equations \eqref{IMHDV} concerning the Gevrey regularity, similar to those obtained in Section \ref{sec:Gev} for equations \eqref{fun_Euler_Voigt}.  However, these will be omitted here because the  concepts are exactly the same, and one only needs to combine the ideas of the proofs of Theorems \ref{thm:gev} and \ref{thm:mag_reg} to recover them.
\end{remark}

\section*{Acknowledgments} We are thankful to the anonymous referees
for their helpful suggestions and insightful comments. This work was
supported in part by the NSF grants no. DMS-0504619 and no.
DMS-0708832, and by the ISF grant no. 120/06.



\end{document}